\documentclass[%
aip,
amsmath,amssymb,
reprint,%
]{revtex4-1}

\usepackage{graphicx}
\usepackage{dcolumn}
\usepackage{bm}

\usepackage{multirow}%
\usepackage{amsmath,amssymb,amsfonts}%
\usepackage{amsthm}%
\usepackage{mathrsfs}%
\usepackage[title]{appendix}%
\usepackage{xcolor}%
\usepackage{textcomp}%
\usepackage{booktabs}%
\usepackage{algorithm}%
\usepackage{algorithmicx}%
\usepackage{algpseudocode}%
\usepackage{listings}%
\usepackage{mhchem}%
\usepackage{float}
\usepackage{subfigure}
\usepackage{tikz}

\usepackage{booktabs}%
\usepackage{tabularx}
\usepackage{makecell}
\usepackage{ragged2e}
\usepackage{lineno}
\usepackage{subfigure}
\usepackage{float}
\usepackage{caption}

\usepackage[utf8]{inputenc}
\usepackage[T1]{fontenc}
\usepackage{mathptmx}
\usepackage{etoolbox}

\theoremstyle{thmstyleone}%
\newtheorem{theorem}{Theorem}[section]
\newtheorem{lemma}[theorem]{Lemma}%
\newtheorem{proposition}[theorem]{Proposition}%
\newtheorem{corollary}[theorem]{Corollary}%

\theoremstyle{thmstyletwo}%
\newtheorem{remark}{Remark}[section]%

\theoremstyle{thmstylethree}%
\newtheorem{definition}{Definition}[section]%

\makeatletter
\def\@email#1#2{%
	\endgroup
	\patchcmd{\titleblock@produce}
	{\frontmatter@RRAPformat}
	{\frontmatter@RRAPformat{\produce@RRAP{*#1\href{mailto:#2}{#2}}}\frontmatter@RRAPformat}
	{}{}
}%
\makeatother
\begin{document}
	
	
	\title{Exact Asymptotics for the Exit Time Probabilities of Scalar Ornstein-Uhlenbeck Bridges}
	\author{Feng Zhao}
	\affiliation{State Key Laboratory of Mechanics and Control for Aerospace Structures, College of Aerospace Engineering, Nanjing University of Aeronautics and Astronautics, 29 Yudao Street, Nanjing 210016, China}%
	
	\author{Yang Li}
	\affiliation{School of Automation, Nanjing University of Science and Technology, 200 Xiaolingwei Street, Nanjing 210094, China}
	
	\author{Jianlong Wang}
	\affiliation{College of Mechanical and Electrical Engineering, Wenzhou University, Chashan University Town, Wenzhou 325035, China}
	
	\author{Xianbin Liu}
	\altaffiliation[\textbf{Authors to whom correspondence should be addressed}]{}
	\email{xbliu@nuaa.edu.cn}
	\author{Dongping Jin}
	\altaffiliation[\textbf{Authors to whom correspondence should be addressed}]{}
	\email{jindp@nuaa.edu.cn}
	
	\affiliation{State Key Laboratory of Mechanics and Control for Aerospace Structures, College of Aerospace Engineering, Nanjing University of Aeronautics and Astronautics, 29 Yudao Street, Nanjing 210016, China}%
	
	
	\begin{abstract}
		This paper aims to derive accurate asymptotic estimates for the exit time probabilities of scalar Ornstein-Uhlenbeck (OU) bridges. The exit time probabilities are expressed as an asymptotic series in powers of a small parameter that characterizes the intensity of the noise inputs. It is shown that the series is valid in certain regions where all its terms are smooth functions. The results enable an accurate evaluation of the probability for a corresponding OU process to escape from a domain before a specified time, provided its initial and terminal states are known.
	\end{abstract}
	
	\maketitle
	
	\begin{quotation}
		An Ornstein-Uhlenbeck (OU) bridge is a continuous-time Gauss-Markov process whose probability distribution is the conditional probability distribution of the corresponding OU process subject to specified initial and terminal conditions. Analyzing the properties of such an OU bridge is conducive to a comprehensive characterization of numerous detailed dynamical features of the corresponding OU process. The present paper derives accurate asymptotic expansions for the exit time probabilities of scalar OU bridges. The results obtained provide a precise estimate for the probability that the corresponding OU process escapes from a domain before a specified time, on the condition that its initial and terminal states are given.
	\end{quotation}
	
	\section{Introduction}\label{Sec01} 
	Fix $T>0$ and $x_T\in\mathbb{R}$. Let $D=(d_1,d_2)\subset\mathbb{R}$ (with $d_1<d_2$) be a non-empty open interval. Consider a scalar Ornstein-Uhlenbeck
	bridge \cite{Barczy_2013,Azze_2024,Yue_2024} $\{x^{\varepsilon}_t\}_{t\in[s,T]}$ terminating at $x^{\varepsilon}_T=x_T$, which is governed by the stochastic differential equation (SDE):
	\begin{equation}\label{eq010101}
		\begin{aligned}
			\mathrm{d}x^{\varepsilon}_t=&b(x^{\varepsilon}_t,t)\mathrm{d}t+\sqrt{\varepsilon}\mathrm{d}w_t,\quad{0}\leq{s}<t<{T},\\
			x^{\varepsilon}_s=&x\in{D},
		\end{aligned}
	\end{equation}
	with
	\begin{equation}\label{eq010102}
		b(x,t)=a_1\left[\frac{\left(x_T+\frac{a_0}{a_1}\right)-\left(x+\frac{a_0}{a_1}\right)\cosh(a_1(T-t))}{\sinh(a_1(T-t))}\right],
	\end{equation}
	where $w_t$ is a one-dimensional standard Wiener process, $a_0\in\mathbb{R}$, $a_1\neq{0}$, and $\varepsilon>0$ is a small parameter. It can be shown that $x^{\varepsilon}$ is a Gauss-Markov process \cite{Nardo_2001,Donofrio_2019,Azze_2025} with mean function
	\begin{equation}\label{eq010103}
		\begin{aligned}
			\mathbb{E}^{\varepsilon}_{x,s}(x^{\varepsilon}_t)=&\left(x+\frac{a_0}{a_1}\right)\frac{\sinh(a_1(T-t))}{\sinh(a_1(T-s))}+\left(x_T+\frac{a_0}{a_1}\right)\\
			&\times\frac{\sinh(a_1(t-s))}{\sinh(a_1(T-s))}-\frac{a_0}{a_1},
		\end{aligned}
	\end{equation}
	and covariance function 
	\begin{equation}\label{eq010104}
		\mathbb{C}\text{ov}(x^{\varepsilon}_{v},x^{\varepsilon}_{t})=\frac{\varepsilon}{a_1}\frac{\sinh(a_1(v-s))\sinh(a_1(T-t))}{\sinh(a_1(T-s))},
	\end{equation}
	for $0\leq{s}\leq{v}\leq{t}<{T}$.
	
	Let $\tau^{\varepsilon}_{x,s}\triangleq\inf\{t>s:x^{\varepsilon}_{t}\notin{D}\}$ denote the first exit time from $D$ of the process $x^{\varepsilon}$, and define the exit time probability function as
	\begin{equation}\label{eq010105}
		q^{\varepsilon}(x,s)\triangleq \mathbb{P}^{\varepsilon}_{x,s}\left(\tau^{\varepsilon}_{x,s}\leq{T}\right).
	\end{equation}
	This paper aims to derive the asymptotic expansion of $q^{\varepsilon}$ in powers of $\varepsilon$, which takes the form of the asymptotic series:
	\begin{equation}\label{eq010106}
		q^{\varepsilon}=\exp\left(-{u}/{\varepsilon}-w\right)\left(1+\varepsilon\psi_1\cdots\varepsilon^{m}\psi_m+o(\varepsilon^{m})\right),
	\end{equation}
	as $\varepsilon\to{0}$, valid at certain points $(x,s)\in{D}\times[0,T)$ where the functions $u$, $w$ and $\psi_m$ are smooth.
	
	The first key result regarding this problem was achieved by Freidlin and Wentzell \cite{Wentzell_1970,Freidlin_2012}. Utilizing the large deviation principle as a foundation, they proved that
	\begin{equation}\label{eq010107}
		\lim_{\varepsilon\to{0}}\varepsilon\ln{q}^{\varepsilon}=-u,
	\end{equation}
	where
	\begin{equation}\label{eq010108}
		u(x,s)\triangleq\inf_{\gamma\in{\Gamma}_{x,s}}I(\gamma)\triangleq\inf_{\gamma\in{\Gamma}_{x,s}}\frac{1}{2}\int_{s}^{T}[\dot{\gamma}(t)-b(\gamma(t),t)]^2\mathrm{d}t,
	\end{equation}
	with the set ${\Gamma}_{x,s}$ defined as
	\begin{equation*}
		\Gamma_{x,s}\triangleq\{\gamma\in\mathcal{C}([s,T],\mathbb{R}):\gamma(s)=x,\, \exists\,t\in(s,T]\, \text{s.t.}\, \gamma(t)\in\partial{D}\}.
	\end{equation*}
	Notably, this result can also be deduced via stochastic control methods \cite{Fleming_1977} and vanishing viscosity techniques \cite{Fleming_1986a}.
	
	The generalization of (\ref{eq010107}) to a full asymptotic series of the form (\ref{eq010106}) is credited to Fleming and James \cite{Fleming_1992}. They provided a PDE-based proof of the asymptotic series by showing that Equation (\ref{eq010106}) holds on compact subsets of certain strongly regular regions $N\subset{D}\times[0,T)$. Their original formulation addressed general non-degenerate diffusion processes without imposing specific restrictions on the form of the drift term $b(x,t)$. In such a general context, however, it cannot be guaranteed that all points in ${D}\times[0,T)$ are strongly regular. Indeed, the strongly regular region may even be empty (see, e.g., Remark 2.3 of [\onlinecite{Baldi_1995}] for an example). This considerably limits the practical utility of this asymptotic expansion.
	
	Notable progress has been made when the drift term $b(x,t)$ takes certain specific forms, as demonstrated in the works of Baldi and colleagues \cite{Baldi_1995,Baldi_2002,Baldi_2015}. A salient example is the Brownian bridge, which corresponds to setting $a_0=0$ and taking $a_1\to{0}$ in (\ref{eq010102}). In this particular instance, the drift term becomes $b(x,t)=\frac{x_T-x}{T-t}$, which is affine in $x$. As a result, the functions $u$, $w$ and $\psi_m$ can all be obtained explicitly. Leveraging this fact, Baldi \cite{Baldi_1995} derived the asymptotics for the exit time probabilities of Brownian bridges, and showed that the resulting series is valid for almost every $(x,t)\in{D}\times[0,T)$. These results have been successfully applied in designing accurate numerical algorithms for simulating Brownian motions, thereby enabling precise calculations of mean first exit times \cite{Baldi_1995}. 

	In this paper, we extend Baldi's result in [\onlinecite{Baldi_1995}] to the case of an Ornstein-Uhlenbeck bridge (i.e., $a_0\in\mathbb{R}$, $a_1\neq{0}$ in Eq. (2)). Our goal is to show that the asymptotic expansion of $q^{\varepsilon}$ likewise holds for almost every $(x,t)\in{D}\times[0,T)$. To this end, the ensuing analysis is organized into two distinct cases, $x_T\notin[d_1,d_2]$ and $x_T\in(d_1,d_2)$, corresponding to the situations where $u\equiv{0}$ and $u\geq{0}$, respectively. For each case, we examine the smoothness of $u$ to determine which points are strongly regular and which are not. The corresponding results are presented in Sections \ref{Sec02} and \ref{Sec03}, respectively. In the limiting case where $a_0=0$ and $a_1\to{0}$, the OU bridge reduces to a Brownian bridge. Consequently, the asymptotics for the exit time probabilities of Brownian bridges can be derived as corollaries of our main theorems, as detailed in Section \ref{Sec04}. Finally, the conclusions of this study are summarized in Section \ref{Sec05}. 
	
	\section{Case A: $x_T\notin[d_1,d_2]$}\label{Sec02}
	For simplicity, we assume that $x_T>d_2>d_1$. The case where $d_2>d_1>x_T$ can be treated similarly.
	
	Let $x^{0}_{x,s}$ denote the solution of the deterministic ordinary differential equation (ODE):
	\begin{equation}\label{eq020101}
		\begin{aligned}
			\mathrm{d}x^{0}_{x,s}(t)=&b(x^{0}_{x,s}(t),t)\mathrm{d}t,\quad{0}\leq{s}<t<{T},\\
			x^{0}_{x,s}(s)=&x\in{D}.
		\end{aligned}
	\end{equation}
	Clearly,
	\begin{equation}\label{eq020102}
		\begin{aligned}
					x^{0}_{x,s}&(t)=\mathbb{E}^{\varepsilon}_{x,s}(x^{\varepsilon}_t)=\left(x+\frac{a_0}{a_1}\right)\frac{\sinh(a_1(T-t))}{\sinh(a_1(T-s))}\\
					&+\left(x_T+\frac{a_0}{a_1}\right)\frac{\sinh(a_1(t-s))}{\sinh(a_1(T-s))}-\frac{a_0}{a_1}, \quad t\in[s,T].
		\end{aligned}
	\end{equation}
	The monotonicity of $x^{0}_{x,s}$ on $[s,T]$ is characterized by the following lemma.
	\begin{lemma}\label{theo0201}
		The deterministic trajectory $x^{0}_{x,s}$ is non-monotonic on $[s,T]$ if
		\begin{equation*}
			\begin{aligned}
				&\left(x_T+\frac{a_0}{a_1}\right)\neq{0}\quad\text{and}\quad\\
				&\frac{\left(x+\frac{a_0}{a_1}\right)}{\left(x_T+\frac{a_0}{a_1}\right)}\in\left(\frac{1}{\cosh(a_1(T-s))},\cosh(a_1(T-s))\right),
			\end{aligned}
		\end{equation*}
		and is monotonic otherwise. A detailed classification of the monotonicity is provided in Table \ref{Table01}.
	\end{lemma}
	\begin{table}[ht]
		\centering
		\caption{Monotonicity of $x^{0}_{x,s}$ on $[s,T]$. Here, $t_1$ is the unique point in $(s,T)$ such that $\dot{x}^{0}_{x,s}(t_1)=0$.}
		\label{Table01}
		\resizebox{\linewidth}{!}{
			\begin{tabular}{c c c}
				\toprule[1pt]
				\text{Conditions} & {{{Sign of} $\dot{x}^{0}_{x,s}$}} & \text{Monotonicity of ${x}^{0}_{x,s}$} \\
				\hline
				{$\left\{\left(x+\frac{a_0}{a_1}\right)>0,\;\left(x_T+\frac{a_0}{a_1}\right)<{0}\right\}$} &{$\dot{x}^{0}_{x,s}(t)<0,\quad\forall t\in[s,T]$} & {$\searrow$} \\
				{$\left\{\left(x+\frac{a_0}{a_1}\right)<0,\;\left(x_T+\frac{a_0}{a_1}\right)>{0}\right\}$} &{$\dot{x}^{0}_{x,s}(t)>0,\quad\forall t\in[s,T]$} & {$\nearrow$} \\
				{$\left\{\left(x+\frac{a_0}{a_1}\right)>0,\;\left(x_T+\frac{a_0}{a_1}\right)={0}\right\}$} &{$\dot{x}^{0}_{x,s}(t)<0,\quad\forall t\in[s,T]$} & {$\searrow$} \\
				{$\left\{\left(x+\frac{a_0}{a_1}\right)=0,\;\left(x_T+\frac{a_0}{a_1}\right)={0}\right\}$} &{$\dot{x}^{0}_{x,s}(t)\equiv{0},\quad\forall t\in[s,T]$} & {\text{Constant}} \\
				{$\left\{\left(x+\frac{a_0}{a_1}\right)<0,\;\left(x_T+\frac{a_0}{a_1}\right)={0}\right\}$} &{$\dot{x}^{0}_{x,s}(t)>0,\quad\forall t\in[s,T]$} & {$\nearrow$} \\
				{$\left\{\left(x+\frac{a_0}{a_1}\right)=0,\;\left(x_T+\frac{a_0}{a_1}\right)>{0}\right\}$} &{$\dot{x}^{0}_{x,s}(t)>0,\quad\forall t\in[s,T]$} & {$\nearrow$} \\
				{$\left\{\left(x+\frac{a_0}{a_1}\right)=0,\;\left(x_T+\frac{a_0}{a_1}\right)<{0}\right\}$} &{$\dot{x}^{0}_{x,s}(t)<0,\quad\forall t\in[s,T]$} & {$\searrow$} \\
				{$\left\{\begin{array}{cc}
						\left(x+\frac{a_0}{a_1}\right)>0,\;\left(x_T+\frac{a_0}{a_1}\right)>{0},\\
						\frac{\left(x+\frac{a_0}{a_1}\right)}{\left(x_T+\frac{a_0}{a_1}\right)}\in\left(0,\frac{1}{\cosh(a_1(T-s))}\right]
					\end{array}\right\}$} &{$\dot{x}^{0}_{x,s}(t)\geq0,\quad\forall t\in[s,T]$} & {$\nearrow$} \\
				
				{$\left\{\begin{array}{cc}
						\left(x+\frac{a_0}{a_1}\right)>0,\;\left(x_T+\frac{a_0}{a_1}\right)>{0},\\
						\frac{\left(x+\frac{a_0}{a_1}\right)}{\left(x_T+\frac{a_0}{a_1}\right)}\in\left(\frac{1}{\cosh(a_1(T-s))},\cosh(a_1(T-s))\right)
					\end{array}\right\}$} &{$\;\dot{x}^{0}_{x,s}(t)\begin{cases}
						\leq{0},\; &\forall t\in[s,t_1]\\
						>0,\; &\forall t\in(t_1,T]
					\end{cases}\;$} & {$\searrow\nearrow$} \\
				
				{$\left\{\begin{array}{cc}
						\left(x+\frac{a_0}{a_1}\right)>0,\;\left(x_T+\frac{a_0}{a_1}\right)>{0},\\
						\frac{\left(x+\frac{a_0}{a_1}\right)}{\left(x_T+\frac{a_0}{a_1}\right)}\in\left[\cosh(a_1(T-s)),\infty\right)
					\end{array}\right\}$} &{$\dot{x}^{0}_{x,s}(t)\leq0,\quad\forall t\in[s,T]$} &{$\searrow$} \\
				
				{$\left\{\begin{array}{cc}
						\left(x+\frac{a_0}{a_1}\right)<0,\;\left(x_T+\frac{a_0}{a_1}\right)<{0},\\
						\frac{\left(x+\frac{a_0}{a_1}\right)}{\left(x_T+\frac{a_0}{a_1}\right)}\in\left(0,\frac{1}{\cosh(a_1(T-s))}\right]
					\end{array}\right\}$} &{$\dot{x}^{0}_{x,s}(t)\leq0,\quad\forall t\in[s,T]$} & {$\searrow$} \\
				
				{$\left\{\begin{array}{cc}
						\left(x+\frac{a_0}{a_1}\right)<0,\;\left(x_T+\frac{a_0}{a_1}\right)<{0},\\
						\frac{\left(x+\frac{a_0}{a_1}\right)}{\left(x_T+\frac{a_0}{a_1}\right)}\in\left(\frac{1}{\cosh(a_1(T-s))},\cosh(a_1(T-s))\right)
					\end{array}\right\}$} &{$\;\dot{x}^{0}_{x,s}(t)\begin{cases}
						\geq{0},\; &\forall t\in[s,t_1]\\
						<0,\; &\forall t\in(t_1,T]
					\end{cases}\;$} & {$\nearrow\searrow$} \\
				
				{$\left\{\begin{array}{cc}
						\left(x+\frac{a_0}{a_1}\right)<0,\;\left(x_T+\frac{a_0}{a_1}\right)<{0},\\
						\frac{\left(x+\frac{a_0}{a_1}\right)}{\left(x_T+\frac{a_0}{a_1}\right)}\in\left[\cosh(a_1(T-s)),\infty\right)
					\end{array}\right\}$} &{$\dot{x}^{0}_{x,s}(t)\geq0,\quad\forall t\in[s,T]$} & {$\nearrow$} \\
				\bottomrule[1pt]
		\end{tabular}}
	\end{table}
	\begin{proof}
		The result follows directly by examining the sign of the derivative
		\begin{equation*}
			\begin{aligned}
				\dot{x}^{0}_{x,s}(t)=&-\left(x+\frac{a_0}{a_1}\right)\frac{a_1\cosh(a_1(T-t))}{\sinh(a_1(T-s))}\\
				&+\left(x_T+\frac{a_0}{a_1}\right)\frac{a_1\cosh(a_1(t-s))}{\sinh(a_1(T-s))},\quad t\in[s,T].
			\end{aligned}
		\end{equation*}
	\end{proof}
	\begin{remark}\label{rem020101}
		When $x_T+\frac{a_0}{a_1}\neq{0}$, two curves in $\mathbb{R}\times[0,T)$ are of particular interest:
		\begin{equation*}
			\text{Cur}_{I}\triangleq\left\{(x,s):x+\frac{a_0}{a_1}=\left({x_T+\frac{a_0}{a_1}}\right)\bigg/{\cosh(a_1(T-s))}\right\},
		\end{equation*}
		\begin{equation*}
			\text{Cur}_{II}\triangleq\left\{(x,s):x+\frac{a_0}{a_1}=\left({x_T+\frac{a_0}{a_1}}\right){\cosh(a_1(T-s))}\right\}.
		\end{equation*}
		These curves possess the following properties:
		\begin{itemize}
			\item $\text{Cur}_{I}=\{(x,s)\in\mathbb{R}\times[0,T):b(x,s)=0\}$, i.e., $\text{Cur}_{I}$ coincides with the nullcline of the vector field $(b(x,s),1)^{\top}$. For any $(x,s)\in\text{Cur}_{I}$, $x^{0}_{x,s}$ satisfies $\dot{x}^{0}_{x,s}(s)=0$ and is monotonic on $[s,T]$.
			\item For any $(x,s)\in\text{Cur}_{II}$, $x^{0}_{x,s}$ is monotonic on $[s,T]$ and satisfies $\dot{x}^{0}_{x,s}(T)=0$. Moreover, the entire integral curve (or the entire graph) $\{(x^{0}_{x,s}(t),t):t\in[s,T)\}$ lies on $\text{Cur}_{II}$, i.e., $(x^{0}_{x,s}(t),t)\in\text{Cur}_{II}$ for all $t\in[s,T)$.
			\item The curves $\text{Cur}_{I}$ and $\text{Cur}_{II}$ act as separatrices. Trajectories starting on opposite sides of these curves exhibit distinct monotonicity behaviors. Specifically, $x^{0}_{x,s}$ is non-monotonic if and only if $(x,s)$ lies in the region between these curves, which corresponds precisely to the condition stated in Lemma \ref{theo0201}, i.e.,
			\begin{equation*}
				\begin{aligned}
					\;\;(x,s)\in\Omega\triangleq\bigg\{&(x,s):\frac{1}{\cosh(a_1(T-s))}<\\
					&\frac{\left(x+{a_0}/{a_1}\right)}{\left(x_T+{a_0}/{a_1}\right)}<\cosh(a_1(T-s))\bigg\}.
				\end{aligned}
			\end{equation*}
			Outside $\Omega$, all trajectories are monotonic.
		\end{itemize}
	\end{remark}
	
	We now proceed to analyze Case A by distinguishing seven subcases based on the relative positions of $-a_0/a_1$, $d_1$, $d_2$, and $x_T$:
	\begin{itemize}
		\item Case $\text{A}_{I}$: $-a_0/a_1>x_T>d_2>d_1$;
		\item Case $\text{A}_{II}$: $-a_0/a_1=x_T>d_2>d_1$;
		\item Case $\text{A}_{III}$: $x_T>-a_0/a_1>d_2>d_1$;
		\item Case $\text{A}_{IV}$: $x_T>-a_0/a_1=d_2>d_1$;
		\item Case $\text{A}_{V}$: $x_T>d_2>-a_0/a_1>d_1$;
		\item Case $\text{A}_{VI}$: $x_T>d_2>-a_0/a_1=d_1$;
		\item Case $\text{A}_{VII}$: $x_T>d_2>d_1>-a_0/a_1$;
	\end{itemize}
	Depending on (i) whether the intersection ${D}\times[0,T)\cap\Omega$ is empty, and (ii) whether there exist points $(x,s)\in{D}\times[0,T)$ such that $\{(x^{0}_{x,s}(t),t):t\in[s,T]\}$ intersects $\partial{D}\times[0,T)$ at more than one point, we further refine this classification as follows:
	\begin{itemize}
		\item Cases $\text{A}_{I}$, $\text{A}_{V}$ and $\text{A}_{VI}$ are each subdivided into two subcases, denoted by $\text{A}_{i,I}$ and $\text{A}_{i,II}$ for $i\in\{I,V,VI\}$.
		\item Cases $\text{A}_{II}$, $\text{A}_{III}$ and $\text{A}_{IV}$ are not subdivided.
		\item Case $\text{A}_{VII}$ is subdivided into three subcases: $\text{A}_{VII,I}$, $\text{A}_{VII,II}$ and $\text{A}_{VII,III}$.
	\end{itemize}
	For each resulting (sub)case, the monotonicity of $x^{0}_{x,s}$ on $[s,T]$ for $(x,s)\in{D}\times[0,T)$ follows directly from Lemma \ref{theo0201}. The detailed characterization is presented below.
	\begin{corollary}\label{theo0202}
		As summarized in Table \ref{Table02}, the following monotonicity properties hold:
		
		(a) In Cases $\text{A}_{I,I}$, $\text{A}_{II}$, $\text{A}_{III}$, $\text{A}_{IV}$, $\text{A}_{V,I}$, $\text{A}_{VI,I}$ and $\text{A}_{VII,I}$, the intersection ${D}\times[0,T)\cap\Omega$ is empty. Hence, $x^{0}_{x,s}$ is monotone increasing on $[s,T]$ for all $(x,s)\in{D}\times[0,T)$.
		
		(b) In Cases $\text{A}_{I,II}$, $\text{A}_{V,II}$, $\text{A}_{VI,II}$, $\text{A}_{VII,II}$ and $\text{A}_{VII,III}$, the intersection is non-empty. Consequently, $x^{0}_{x,s}$ is monotone increasing on $[s,T]$ for any $(x,s)\in{D}\times[0,T)\setminus\Omega$, and is non-monotonic for any $(x,s)\in{D}\times[0,T)\cap\Omega$.
	\end{corollary}
	\begin{table}[ht]
		\centering
		\caption{Monotonicity of $\{x^{0}_{x,s}(t):t\in[s,T]\}$ for $(x,s)\in{D}\times[0,T)$ in Case A.}
		\label{Table02}
		\resizebox{\linewidth}{!}{
			\begin{tabular}{c c c c}
				\toprule[1pt]
				\text{Cases} & \text{Subcases} & \text{Monotonicity of $x^{0}_{x,s}$} \\
				\hline
				\multirow{3}{*}{$\text{A}_{I}$} & {$\text{A}_{I,I}\;:\;\left\{d_2+\frac{a_0}{a_1}\leq\left(x_T+\frac{a_0}{a_1}\right)\cosh(a_1T)\right\}$}  & {$\;\nearrow,\qquad\forall (x,s)\in{D}\times[0,T)\;\;\;\;\;\;$} \\
				& {$\text{A}_{I,II}:\;\left\{d_2+\frac{a_0}{a_1}>\left(x_T+\frac{a_0}{a_1}\right)\cosh(a_1T)\right\}$} &
				{$\;\begin{cases}
						\nearrow\searrow, &\forall (x,s)\in{D}\times[0,T)\cap\Omega\\
						\nearrow, &\forall (x,s)\in{D}\times[0,T)\setminus\Omega
					\end{cases}\;$} \\
				\hline
				{$\text{A}_{II},\text{A}_{III},\text{A}_{IV}$} & ------ & {$\;\nearrow,\qquad\forall (x,s)\in{D}\times[0,T)\;\;\;\;\;\;$} \\
				\hline
				\multirow{3}{*}{$\text{A}_{V},\text{A}_{VI}$} & {$\text{A}_{V,I},\;\text{A}_{VI,I}\;:\;\,\left\{d_2+\frac{a_0}{a_1}\leq\left(x_T+\frac{a_0}{a_1}\right)/\cosh(a_1T)\right\}$}  & {$\;\nearrow,\qquad\forall (x,s)\in{D}\times[0,T)\;\;\;\;\;\;$}  \\
				& {$\text{A}_{V,II},\text{A}_{VI,II}:\;\left\{d_2+\frac{a_0}{a_1}>\left(x_T+\frac{a_0}{a_1}\right)/\cosh(a_1T)\right\}$} &
				{$\;\begin{cases}
						\searrow\nearrow, &\forall (x,s)\in{D}\times[0,T)\cap\Omega\\
						\nearrow, &\forall (x,s)\in{D}\times[0,T)\setminus\Omega
					\end{cases}\;$}  \\
				\hline
				\multirow{7}{*}{$\text{A}_{VII}$} & {$\text{A}_{VII,I}\;\;:\;\left\{d_2+\frac{a_0}{a_1}\leq\left(x_T+\frac{a_0}{a_1}\right)/\cosh(a_1T)\right\}\;\;\;\;\;\;\;\,$}  & {$\;\nearrow,\qquad\forall (x,s)\in{D}\times[0,T)\;\;\;\;\;\;$}  \\
				& {$\text{A}_{VII,II}:\;\left\{\begin{array}{cc}
						d_2+\frac{a_0}{a_1}>\left(x_T+\frac{a_0}{a_1}\right)/\cosh(a_1T)\\
						d_1+\frac{a_0}{a_1}\leq\left(x_T+\frac{a_0}{a_1}\right)/\cosh(a_1T)
					\end{array}\right\}\;\;$} &
				{$\;\begin{cases}
						\searrow\nearrow, &\forall (x,s)\in{D}\times[0,T)\cap\Omega\\
						\nearrow, &\forall (x,s)\in{D}\times[0,T)\setminus\Omega
					\end{cases}\;$} \\
				& {$\text{A}_{VII,III}:\;\left\{d_1+\frac{a_0}{a_1}>\left(x_T+\frac{a_0}{a_1}\right)/\cosh(a_1T)\right\}\;\;\;\;\;\;\;\;$}
				&
				{$\;\begin{cases}
						\searrow\nearrow, &\forall (x,s)\in{D}\times[0,T)\cap\Omega\\
						\nearrow, &\forall (x,s)\in{D}\times[0,T)\setminus\Omega
					\end{cases}\;$} \\
				\bottomrule[1pt]
		\end{tabular}}
	\end{table}
	
	Let $\tau^{0}_{x,s}\triangleq\inf\{t>s:x^{0}_{x,s}(t)\notin{D}\}$ be the first exit time of $x^{0}_{x,s}$ from $D$. The following result can be deduced directly from Corollary \ref{theo0202}.
	\begin{corollary}\label{theo0203}
		(a) In Cases $\text{A}_{I}$, $\text{A}_{II}$, $\text{A}_{III}$, $\text{A}_{IV}$, $\text{A}_{V}$, $\text{A}_{VI}$, $\text{A}_{VII,I}$, and $\text{A}_{VII,II}$, for any $(x,s)\in{D}\times[0,T)$, $\{(x^{0}_{x,s}(t),t):t\in[s,T]\}$ intersects $\partial{D}\times[0,T)$ transversely at a unique point $(d_2,\tau^{0}_{x,s})$. The first exit time $\tau^{0}_{x,s}$ is the unique time in $(s,T)$ satisfying $x^{0}_{x,s}(\tau^{0}_{x,s})=d_2$.
		
		(b) In contrast, Case $\text{A}_{VII,III}$ exhibits different behavior. Let $t_2$ be the unique time in $(0,T)$ satisfying
		\begin{equation*}
			d_1+\frac{a_0}{a_1}=\left({x_T+\frac{a_0}{a_1}}\right)\bigg/{\cosh(a_1(T-t_2))},
		\end{equation*}
		and define the critical trajectory $x^{0,*}$ as
		\begin{equation*}
			\begin{aligned}
				x^{0,*}(t)\triangleq& x^{0}_{d_1,t_2}(t)\\
				=&\left({d_1+\frac{a_0}{a_1}}\right)\cosh(a_1(t_2-t))-\frac{a_0}{a_1},\quad t\in[0,T].
			\end{aligned}
		\end{equation*}
		The domain $D\times[0,T)$ is partitioned into three distinct subdomains:
		\begin{equation*}
			\Sigma_1\triangleq\{(x,s)\in{D}\times[0,T)\cap\Omega:x<x^{0,*}(s),\;s\in[0,T)\},
		\end{equation*}
		\begin{equation*}
			\Sigma_2\triangleq\{(x,s)\in{D}\times[0,T)\cap\Omega:x=x^{0,*}(s),\;s\in[0,T)\},
		\end{equation*}
		\begin{equation*}
			\Sigma_3\triangleq{D}\times[0,T)\setminus(\Sigma_1\cup\Sigma_2).
		\end{equation*}
		Then 
		\begin{itemize}
			\item For any $(x,s)\in\Sigma_1$, $\{(x^{0}_{x,s}(t),t):t\in[s,T]\}$ exits $D\times[0,T)$ transversely through $\{d_1\}\times[0,T)$ at time $\tau^{0}_{x,s}$, re-enters $D\times[0,T)$ at a later time, and finally exits again transversely through $\{d_2\}\times[0,T)$ at another time. The first exit time $\tau^{0}_{x,s}$ is the smallest solution of $x^{0}_{x,s}(\tau^{0}_{x,s})=d_1$ in $(s,T)$.
			\item For any $(x,s)\in\Sigma_2$, $\{(x^{0}_{x,s}(t),t):t\in[s,T]\}$ exits $D\times[0,T)$ tangentially through $\{d_1\}\times[0,T)$ at time $\tau^{0}_{x,s}$ and, after re-entering, exits through $\{d_2\}\times[0,T)$ transversely at a later time. The first exit time $\tau^{0}_{x,s}$ is the unique solution of $x^{0}_{x,s}(\tau^{0}_{x,s})=d_1$ in $(s,T)$, corresponding to the tangential exit point. That is, $\tau^{0}_{x,s}=t_2$.
			\item For any $(x,s)\in\Sigma_3$, $\{(x^{0}_{x,s}(t),t):t\in[s,T]\}$ intersects $\partial{D}\times[0,T)$ transversely at a unique point $(d_2,\tau^{0}_{x,s})$, where $\tau^{0}_{x,s}$ is the unique solution of $x^{0}_{x,s}(\tau^{0}_{x,s})=d_2$ in $(s,T)$.
			\item Moreover, as illustrated in FIG. \ref{fig01d}, the first exit time $\tau^{0}_{x,s}$ is a smooth function of $(x,s)$ within $\Sigma_1\cup\Sigma_3$, but exhibits a discontinuity when $(x,s)$ crosses $\Sigma_2$ transversely.
		\end{itemize}
	\end{corollary}
	\begin{proof}
		The claims follow directly by examining the direction of the vector field $(b(x,t),1)^{\top}$ in $D\times[0,T)$ together with the monotonicity properties of $x^{0}_{x,s}$ established in Corollary \ref{theo0202}. For clarity, a graphical illustration for Case $\text{A}_{VII}$ is provided in FIG. \ref{fig01}. The results for the remaining cases can be established using analogous reasoning.
	\end{proof}
	\begin{figure}[h]
		\centering
		\subfigure{\begin{minipage}[b]{.3\linewidth}
				\centering
				\includegraphics[scale=0.3]{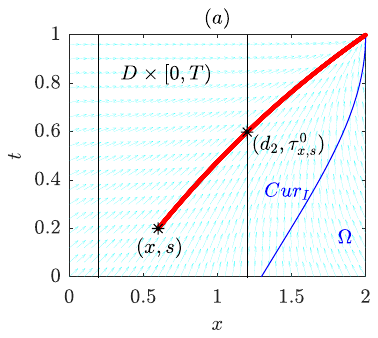}
			\end{minipage}
			\label{fig01a}
		\quad\quad\quad\quad\quad}
		\subfigure{\begin{minipage}[b]{.3\linewidth}
				\centering
				\includegraphics[scale=0.3]{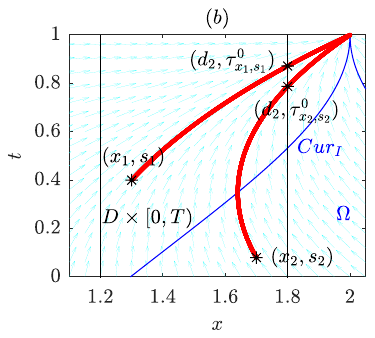}
			\end{minipage}
			\label{fig01b}
		\quad\quad\quad\quad\quad}
		
		\subfigure{\begin{minipage}[b]{.3\linewidth}
				\centering
				\includegraphics[scale=0.3]{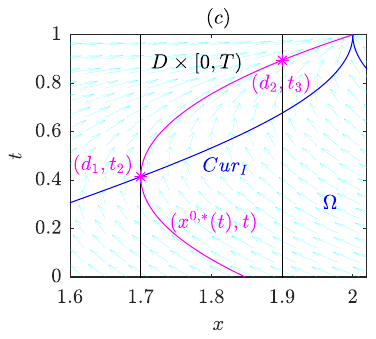}
			\end{minipage}
			\label{fig01c}
		\quad\quad\quad\quad\quad}
		\subfigure{\begin{minipage}[b]{.3\linewidth}
				\centering
				\includegraphics[scale=0.3]{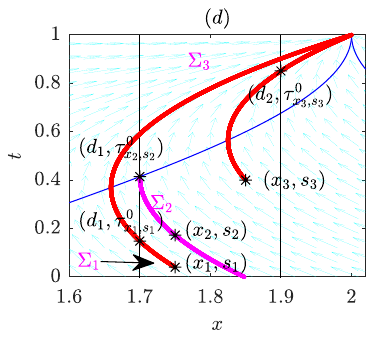}
			\end{minipage}
			\label{fig01d}
		\quad\quad\quad\quad\quad}
		\caption{The vector field $(b(x,t),1)^{\top}$ and the intersection between $\{(x^{0}_{x,s}(t),t):t\in[s,T]\}$ and $\partial{D}\times[0,T)$ for Case $\text{A}_{VII}$. Subfigures (a) and (b) correspond to Cases $\text{A}_{VII,I}$ and $\text{A}_{VII,II}$, respectively, while (c) and (d) illustrate an example for Case $\text{A}_{VII,III}$. The parameters are $a_0=0$, $a_1=1$, $T=1$, $x_T=2$ with (a) $d_1=0.2$, $d_2=1.2$; (b) $d_1=1.2$, $d_2=1.8$; (c)(d) $d_1=1.7$, $d_2=1.9$.}
		\label{fig01}
	\end{figure}
	\begin{remark}\label{rem020102}
		It is important to note that $\tau^{0}_{x,s}$ is invariant along $(x^{0}_{x,s}(t),t)$. More precisely, 
		\begin{equation*}
			\tau^{0}_{x^{0}_{x,s}(t),t}=\tau^{0}_{x,s},\quad\text{for all}\;t\in[s,\tau^{0}_{x,s}].
		\end{equation*}
		It is also clear that $\tau^{0}_{x,s}\in(s,T)$ for every $(x,s)\in{D}\times[0,T)$.
	\end{remark}
	
	These results allow us to calculate the function $u$ explicitly, as described below.
	\begin{corollary}\label{theo0204}
		$u(x,s)\equiv{0}$ for all $(x,s)\in{D}\times[0,T)$ in every case.
	\end{corollary}
	\begin{proof}
		By definition, $u(x,s)\geq{0}$. To show that $u\equiv{0}$, it suffices to prove the reverse inequality. This is achieved by the following chain of (in)equalities:
		\begin{equation*}
			\begin{aligned}
				u(x,s)&\leq\frac{1}{2}\int_{s}^{T}\left[\dot{x}^{0}_{x,s}(t)-b({x}^{0}_{x,s}(t),t)\right]^{2}\mathrm{d}t\\
				&=\frac{1}{2}\int_{s}^{\tau^{0}_{x,s}}\left[\dot{x}^{0}_{x,s}(t)-b({x}^{0}_{x,s}(t),t)\right]^{2}\mathrm{d}t\\
				&=0.
			\end{aligned}
		\end{equation*}
	\end{proof}
	
	Applying the Fleming-James theorem (Theorem \ref{theoA102} in Appendix \ref{SecA1}) to these cases yields the following result.
	\begin{proposition}\label{theo0205}
		(a) In Cases $\text{A}_{I}$, $\text{A}_{II}$, $\text{A}_{III}$, $\text{A}_{IV}$, $\text{A}_{V}$, $\text{A}_{VI}$, $\text{A}_{VII,I}$ and $\text{A}_{VII,II}$, we have
		\begin{equation}\label{eq020103}
			q^{\varepsilon}(x,s)=1+o(\varepsilon^{m}),\quad \text{as}\; \varepsilon\to{0},
		\end{equation}
		for every $m\geq{0}$, uniformly on compact subsets of ${D}\times[0,T)$.
		
		(b) In Case $\text{A}_{VII,III}$, Equation (\ref{eq020103}) holds for every $m\geq{0}$, uniformly on compact subsets of $\Sigma_1\cup\Sigma_3$.
	\end{proposition}
	\begin{proof}
		Following a similar graphical approach, we illustrate the proof for Case $\text{A}_{VII}$. The remaining cases are analogous.
				
		In Cases $\text{A}_{VII,I}$ and $\text{A}_{VII,II}$, for any $T'\in(0,T)$, define 
		\begin{equation*}
			N_{T'}\triangleq\{(x,s)\in{D}\times[0,T):\tau^{0}_{x,s}\in(0,T')\}.
		\end{equation*}
		As indicated in FIGs. \ref{fig02a} and \ref{fig02b}, the region $N_{T'}$ satisfies Assumption (A) of Appendix \ref{SecA1}. Applying the method of characteristics (Remark \ref{remA101}) gives
		\begin{equation*}
			w(x,s)\equiv{0},\quad\forall\,(x,s)\in{N_{T'}},
		\end{equation*}
		\begin{equation*}
			\psi_m\equiv{0},\quad\forall\,(x,s)\in{N_{T'}},\;m=1,2,\cdots.
		\end{equation*}
		Hence, by Theorem \ref{theoA102}, Equation (\ref{eq020103}) holds uniformly on every compact subset of $N_{T'}$. 
		
		Moreover, for any compact subset $K\subset{D}\times[0,T)$, we can choose $T'\in(0,T)$ such that $K\subset{N}_{T'}$. Therefore, the asymptotic series (\ref{eq020103}) is valid uniformly on all compact subsets of ${D}\times[0,T)$.
		
		For Case $\text{A}_{VII,III}$ and any $T'\in(t_2,T)$, we define
		\begin{equation*}
			N^{(1)}_{T'}\triangleq\Sigma_1,
		\end{equation*}
		and
		\begin{equation*}
			N^{(2)}_{T'}\triangleq\{(x,s)\in\Sigma_3:\tau^{0}_{x,s}\in(0,T')\}.
		\end{equation*}
		As depicted in FIG. \ref{fig02c}, each $N^{(i)}_{T'}$ ($i\in\{1,2\}$) satisfies Assumption (A). Consequently,
		\begin{equation*}
			w(x,s)\equiv{0},\quad\forall\;(x,s)\in{N}^{(1)}_{T'}\cup{N}^{(2)}_{T'},
		\end{equation*}
		\begin{equation*}
			\psi_m\equiv{0},\quad\forall\;(x,s)\in{N}^{(1)}_{T'}\cup{N}^{(2)}_{T'},\;m=1,2,\cdots.
		\end{equation*}
		Because every compact subset $K\subset\Sigma_1\cup\Sigma_3$ is contained in ${N}^{(1)}_{T'}\cup{N}^{(2)}_{T'}$ for a suitable $T'\in(t_2,T)$, Equation (\ref{eq020103}) holds uniformly on all compact subsets of $\Sigma_1\cup\Sigma_3$. This completes the proof.
	\end{proof}
	\begin{figure}[h]
		\centering
		\subfigure{\begin{minipage}[b]{.3\linewidth}
				\centering
				\includegraphics[scale=0.3]{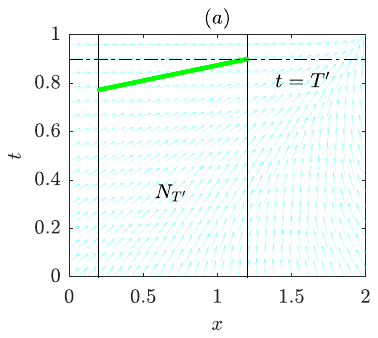}
			\end{minipage}
			\label{fig02a}
		\quad\quad\quad\quad\quad}
		\subfigure{\begin{minipage}[b]{.3\linewidth}
				\centering
				\includegraphics[scale=0.3]{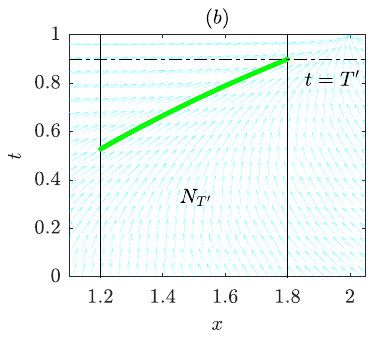}
			\end{minipage}
			\label{fig02b}
		\quad\quad\quad\quad\quad}
		
		\subfigure{\begin{minipage}[b]{.3\linewidth}
				\centering
				\includegraphics[scale=0.3]{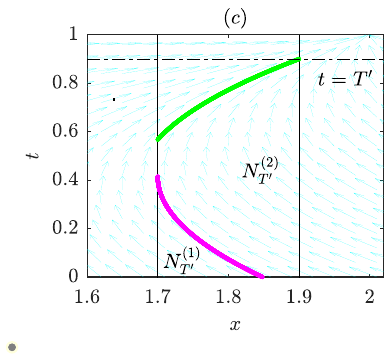}
			\end{minipage}
			\label{fig02c}
		\quad\quad\quad\quad\quad}
		\caption{Strongly regular regions for Case $\text{A}_{VII}$. Subfigures (a) and (b) correspond to Cases $\text{A}_{VII,I}$ and $\text{A}_{VII,II}$, respectively, while (c) illustrates an example for Case $\text{A}_{VII,III}$. In each subfigure, the solid green line represents a segment of the trajectory $\{(x^{0}_{x,s}(t),t):t\in[0,T]\}$ that satisfies $x^{0}_{x,s}(T')=d_2$, and the solid magenta line denotes a segment of the critical trajectory $\{(x^{0,*}(t),t):t\in[0,T]\}$. The parameters are $a_0=0$, $a_1=1$, $T=1$, $x_T=2$, $T'=0.9$ with (a) $d_1=0.2$, $d_2=1.2$; (b) $d_1=1.2$, $d_2=1.8$; (c) $d_1=1.7$, $d_2=1.9$.}
		\label{fig02}
	\end{figure}
	\begin{remark}\label{rem020103}
		(a) Observe that
		\begin{equation*}
			\mathbb{E}^{\varepsilon}_{x,s}\left(x^{\varepsilon}_{T}\right)=x_T\quad\text{and}\quad\mathbb{V}\text{ar}\left(x^{\varepsilon}_{T}\right)=\mathbb{C}\text{ov}\left(x^{\varepsilon}_{T},x^{\varepsilon}_{T}\right)=0.
		\end{equation*}
		It follows that $x^{\varepsilon}_T=x_T$ almost surely. Combined with the fact that the sample path $x^{\varepsilon}$ is almost surely continuous on $[s,T]$ (i.e., $x^{\varepsilon}\in\mathcal{C}([s,T],\mathbb{R})$ a.s.), we conclude that if $x_T\notin{D}$, then for any $(x,s)\in{D}\times[0,T)$, $x^{\varepsilon}$ must exit $D$ at some time $t\in[s,T]$ almost surely. This implies
		\begin{equation}\label{eq020104}
			q^{\varepsilon}(x,s)\equiv{1}, \quad\forall\;(x,s)\in{D}\times[0,T),\; \text{provided}\; x_T\notin{D}.
		\end{equation}
		Therefore, the conclusion of Proposition \ref{theo0205} holds trivially in Case A.

		(b) The underlying detailed dynamics, however, are far from trivial. By the law of large numbers (see, e.g., [\onlinecite[Chapter 2, Theorem 1.2.]{Freidlin_2012}]), for every $T'\in(s,T)$ and $\delta>0$,
		\begin{equation}\label{eq020105}
			\lim_{\varepsilon\to{0}}\mathbb{P}^{\varepsilon}_{x,s}\left(\sup_{t\in[s,T']}\left\vert{x}^{\varepsilon}_t-{x}^{0}_{x,s}(t)\right\vert>\delta\right)=0.
		\end{equation}
		That is, for sufficiently small $\varepsilon>0$, nearly every sample path stays within a $\delta$-tube neighborhood of the deterministic trajectory $x^{0}_{x,s}$. Consequently, when $x_T>d_2>d_1$, the following properties hold (cf. [\onlinecite[Chapter 2, Theorem 2.3.]{Freidlin_2012}]):
		\begin{itemize}
			\item In Cases $\text{A}_{I}$, $\text{A}_{II}$, $\text{A}_{III}$, $\text{A}_{IV}$, $\text{A}_{V}$, $\text{A}_{VI}$, $\text{A}_{VII,I}$, and $\text{A}_{VII,II}$, for any $(x,s)\in{D}\times[0,T)$ and $\eta>0$,
			\begin{equation}\label{eq020106}
				\lim_{\varepsilon\to{0}}\mathbb{P}^{\varepsilon}_{x,s}\left(\left\vert\tau^{\varepsilon}_{x,s}-\tau^{0}_{x,s}\right\vert>\eta\right)={0},
			\end{equation}
			\begin{equation}\label{eq020107}
				\lim_{\varepsilon\to{0}}\mathbb{P}^{\varepsilon}_{x,s}\left({x}^{\varepsilon}_{\tau^{\varepsilon}_{x,s}}=x^{0}_{x,s}(\tau^{0}_{x,s})\right)=1.
			\end{equation}
			In other words, for sufficiently small $\varepsilon$, the sample path $x^{\varepsilon}$ is overwhelmingly likely to exit $D$ through the boundary point $\{x^{0}_{x,s}(\tau^{0}_{x,s})=d_2\}$, and at a time close to the deterministic exit time $\tau^{0}_{x,s}$. 
			
			\item In Case $\text{A}_{VII,III}$, Equations (\ref{eq020106}) and (\ref{eq020107}) remain valid for all $(x,s)\in\Sigma_1\cup\Sigma_3$ and $\eta>0$. Note that $x^{0}_{x,s}(\tau^{0}_{x,s})=d_1$ for $(x,s)\in\Sigma_1$, while $x^{0}_{x,s}(\tau^{0}_{x,s})=d_2$ for $(x,s)\in\Sigma_3$. This indicates that when $(x,s)$ crosses $\Sigma_2$ transversely, discontinuities of the first exit position ${x}^{\varepsilon}_{\tau^{\varepsilon}_{x,s}}$ and the first exit time $\tau^{\varepsilon}_{x,s}$ emerge.
			
			\item For Case $\text{A}_{VII,III}$ with $(x,s)\in\Sigma_2$, Equation (\ref{eq020105}) implies that for sufficiently small $\varepsilon>0$,
			\begin{equation*}
				\mathbb{P}^{\varepsilon}_{x,s}\left({x}^{\varepsilon}_{\tau^{\varepsilon}_{x,s}}=d_1\right)\simeq\frac{1}{2},\quad\mathbb{P}^{\varepsilon}_{x,s}\left({x}^{\varepsilon}_{\tau^{\varepsilon}_{x,s}}=d_2\right)\simeq\frac{1}{2}.
			\end{equation*}
			 That is to say, nearly half of the sample paths exit $D$ through $\{d_1\}$ at a time close to $t_2$. After oscillating in the vicinity of $\{d_1\}$, these paths re-enter $D$ and ultimately cross $\{d_2\}$ at a later time close to $t_3$. The remaining half of the sample paths, which do not exit at a time near $t_2$, are more likely to exit $D$ directly through $\{d_2\}$ at a time close to $t_3$. Here, $t_3$ is defined as the unique time in $(t_2,T)$ satisfying $x^{0,*}(t_3)=d_2$, i.e.,
			 \begin{equation*}
			 	t_3\triangleq\lim_{(x,s)\in\Sigma_3,\;(x,s)\to(d_1,t_2)}\tau^{0}_{x,s}.
			 \end{equation*}
		\end{itemize}
	\end{remark}
	
	\section{Case B: $x_T\in(d_1,d_2)$}\label{Sec03}
	For simplicity, we treat the following four subcases under Case B:
	\begin{itemize}
		\item Case $\text{B}_{I}$: $-a_0/a_1>d_2>x_T>d_1$;
		\item Case $\text{B}_{II}$: $-a_0/a_1=d_2>x_T>d_1$;
		\item Case $\text{B}_{III}$: $d_2>-a_0/a_1>x_T>d_1$;
		\item Case $\text{B}_{IV}$: $d_2>-a_0/a_1=x_T>d_1$;
	\end{itemize}
	The remaining cases follow similarly.
	
	Depending on whether $\{(x^{0}_{x,s}(t),t):t\in[s,T]\}$ intersects $\partial{D}\times[0,T)$, Case $\text{B}_{I}$ splits into two subcases: $\text{B}_{I,I}$ and $\text{B}_{I,II}$. The results below parallel those in Section \ref{Sec02} (Corollaries \ref{theo0202}, \ref{theo0203}, \ref{theo0204}, and Proposition \ref{theo0205}). The proofs are omitted due to their similarity.
	\begin{corollary}\label{theo030101}
		As shown in Table \ref{Table03}, the following monotonicity properties are true:
		
		(a) In Cases $\text{B}_{I}$, $\text{B}_{II}$, and $\text{B}_{III}$, the condition $d_1<x_T<d_2$ ensures that ${D}\times[0,T)\cap\Omega$ is non-empty. Consequently, $x^{0}_{x,s}$ is monotone on $[s,T]$ for any $(x,s)\in{D}\times[0,T)\setminus\Omega$ and non-monotonic for any $(x,s)\in{D}\times[0,T)\cap\Omega$.
		
		(b) For Case $\text{B}_{IV}$, the equality $x_T+\frac{a_0}{a_1}=0$ gives $\Omega=\emptyset$. Therefore, $x^{0}_{x,s}$ is monotone on $[s,T]$ for all $(x,s)\in{D}\times[0,T)$.
	\end{corollary}
	\begin{table}[ht]
		\centering
		\caption{Monotonicity of $\{x^{0}_{x,s}(t):t\in[s,T]\}$ for $(x,s)\in{D}\times[0,T)$ in Case B.}
		\label{Table03}
		\resizebox{\linewidth}{!}{
			\begin{tabular}{c c c c}
				\toprule[1pt]
				\text{Cases} & \text{Subcases} & \text{Monotonicity of $x^{0}_{x,s}$} \\
				\hline
				\multirow{3}{*}{$\text{B}_{I}$} & {$\text{B}_{I,I}:\;\left\{d_2+\frac{a_0}{a_1}\geq\left(x_T+\frac{a_0}{a_1}\right)/\cosh(a_1T)\right\}$}  & {$\;\begin{cases}
						\nearrow\searrow, &\forall (x,s)\in{D}\times[0,T)\cap\Omega\\
						\nearrow\;\text{or}\;\searrow, &\forall (x,s)\in{D}\times[0,T)\setminus\Omega
					\end{cases}\;$} \\
				& {$\text{B}_{I,II}:\;\left\{d_2+\frac{a_0}{a_1}<\left(x_T+\frac{a_0}{a_1}\right)/\cosh(a_1T)\right\}$} &
				{$\;\begin{cases}
						\nearrow\searrow, &\forall (x,s)\in{D}\times[0,T)\cap\Omega\\
						\nearrow\;\text{or}\;\searrow, &\forall (x,s)\in{D}\times[0,T)\setminus\Omega
					\end{cases}\;$} \\
				\hline
				{$\text{B}_{II},\text{B}_{III}$} & ------ & {$\;\begin{cases}
						\nearrow\searrow, &\forall (x,s)\in{D}\times[0,T)\cap\Omega\\
						\nearrow\;\text{or}\;\searrow, &\forall (x,s)\in{D}\times[0,T)\setminus\Omega
					\end{cases}\;$} \\
				\hline
				{$\text{B}_{IV}$} & ------ & {$\;\begin{cases}
						\text{Constant}, &\text{if}\; x=x_T=-a_0/a_1\;\;\;\;\;\;\;\!\\
						\nearrow\;\text{or}\;\searrow, &\text{otherwise}
					\end{cases}\;$} \\
				\bottomrule[1pt]
		\end{tabular}}
	\end{table}
	\begin{corollary}\label{theo030102}
		(a) In Cases $\text{B}_{I,I}$, $\text{B}_{II}$, $\text{B}_{III}$, and $\text{B}_{IV}$, for any $(x,s)\in{D}\times[0,T)$, we have $x^{0}_{x,s}(t)\in{D}$ for all $t\in[s,T]$.
		
		(b) In contrast, Case $\text{B}_{I,II}$ exhibits distinct behavior. Let $t_4$ be the unique time in $(0,T)$ satisfying
		\begin{equation*}
			d_2+\frac{a_0}{a_1}=\left(x_T+\frac{a_0}{a_1}\right)\bigg/\cosh(a_1(T-t_4)),
		\end{equation*}
		and define another critical trajectory $x^{0,**}$ as
		\begin{equation*}
			\begin{aligned}
				x^{0,**}(t)\triangleq&{x}^{0}_{d_2,t_4}(t)\\
				=&\left({d_2+\frac{a_0}{a_1}}\right)\cosh(a_1(t_4-t))-\frac{a_0}{a_1},\quad t\in[0,T].
			\end{aligned}
		\end{equation*}
		The domain $D\times[0,T)$ can be partitioned into three distinct subdomains:
		\begin{equation*}
			\Lambda_1\triangleq\{(x,s)\in{D}\times[0,T)\cap\Omega:x>x^{0,**}(s),\;s\in[0,T)\},
		\end{equation*}
		\begin{equation*}
			\Lambda_2\triangleq\{(x,s)\in{D}\times[0,T)\cap\Omega:x=x^{0,**}(s),\;s\in[0,T)\},
		\end{equation*}
		\begin{equation*}
			\Lambda_3\triangleq{D}\times[0,T)\setminus(\Lambda_1\cup\Lambda_2).
		\end{equation*}
		Then 
		\begin{itemize}
			\item For any $(x,s)\in\Lambda_1$, $\{(x^{0}_{x,s}(t),t):t\in[s,T]\}$ exits $D\times[0,T)$ transversely through $\{d_2\}\times[0,T)$ at time $\tau^{0}_{x,s}$, re-enters the domain at a later time, and remains within $D\times[0,T)$ until reaching $(x_T,T)$. The first exit time $\tau^{0}_{x,s}$ is the smallest solution of $x^{0}_{x,s}(\tau^{0}_{x,s})=d_2$ in $(s,T)$.
			\item For any $(x,s)\in\Lambda_2$, $\{(x^{0}_{x,s}(t),t):t\in[s,T]\}$ exits $D\times[0,T)$ tangentially through $\{d_2\}\times[0,T)$ at time $\tau^{0}_{x,s}$ and, after re-entering, remains inside the domain until it reaches $(x_T,T)$. The first exit time $\tau^{0}_{x,s}$ is the unique solution of $x^{0}_{x,s}(\tau^{0}_{x,s})=d_2$ in $(s,T)$, corresponding to the tangential exit point, i.e., $\tau^{0}_{x,s}=t_4$.
			\item For any $(x,s)\in\Lambda_3$, $\{(x^{0}_{x,s}(t),t):t\in[s,T]\}$ stays inside $D\times[0,T)$ until reaching $(x_T,T)$. (In this case, we adopt the convention $\tau^{0}_{x,s}=+\infty$ so that the first exit time $\tau^{0}_{x,s}$ is well-defined.)
		\end{itemize}
	\end{corollary}
	\begin{proposition}\label{theo030103}
		In Case $\text{B}_{I,II}$, we have $u(x,s)\equiv{0}$ for all $(x,s)\in\Lambda_1\cup\Lambda_2$. Consequently, $\Lambda_1$ satisfies Assumption (A), and by the Fleming-James theorem,
		\begin{equation}\label{eq030101}
			q^{\varepsilon}(x,s)=1+o(\varepsilon^{m}),\quad\text{as}\;\varepsilon\to{0},
		\end{equation}
		for every $m\geq{0}$, uniformly on every compact subset of $\Lambda_1$.
	\end{proposition}
	\begin{remark}\label{rem030101}
		The infinitesimal term $o(\varepsilon^{m})$ in (\ref{eq030101}) cannot be omitted, because the asymptotic series in Theorem \ref{theoA102} is essentially smooth but not analytic. To see this, let $\Gamma^{\text{C}}_{x,s}$ denote the complement of $\Gamma_{x,s}$, i.e.,
		\begin{equation*}
			\begin{aligned}
				\Gamma^{\text{C}}_{x,s}&\triangleq\{\gamma\in\mathcal{C}([s,T],\mathbb{R}):\gamma(s)=x\}\setminus\Gamma_{x,s}\\
				&=\{\gamma\in\mathcal{C}([s,T],\mathbb{R}):\gamma(s)=x,\,\gamma(t)\in{D},\forall\,t\in[s,T]\}.
			\end{aligned}
		\end{equation*}
		Then
		\begin{equation*}
			\begin{aligned}
				\mathbb{P}^{\varepsilon}_{x,s}\left(\tau^{\varepsilon}_{x,s}\leq{T}\right)&=1-\mathbb{P}^{\varepsilon}_{x,s}\left(\tau^{\varepsilon}_{x,s}>{T}\right)\\
				&=1-\mathbb{P}^{\varepsilon}_{x,s}\left(x^{\varepsilon}\in\Gamma^{\text{C}}_{x,s}\right).
			\end{aligned}
		\end{equation*}
		By Borovkov's characterization of the large deviation principle (cf. [\onlinecite[Chapter 3, Theorem 3.4]{Freidlin_2012}]), for each $(x,s)\in\Lambda_1$, there exist constants $C_1>C_2>0$ (determined by the minimum $\inf_{\gamma\in{\Gamma}^{\text{C}}_{x,s}}I(\gamma)$) such that for all sufficiently small $\varepsilon$,
		\begin{equation*}
			\exp\left(-{C_1}/{\varepsilon}\right)\leq\mathbb{P}^{\varepsilon}_{x,s}\left(x^{\varepsilon}\in\Gamma^{\text{C}}_{x,s}\right)\leq\exp\left(-{C_2}/{\varepsilon}\right).
		\end{equation*}
		This implies
		\begin{equation*}
			q^{\varepsilon}(x,s)\simeq{1}-\exp\left(-{C}/{\varepsilon}\right)
		\end{equation*}
		for some $C>0$. Since the function $\exp\left(-{C}/{\varepsilon}\right)$ is smooth but non-analytic in $\varepsilon$ at $\varepsilon=0$, it cannot be expanded in a convergent Taylor series. Consequently, although Equation (\ref{eq030101}) holds for every $m\geq{0}$, the term $o(\varepsilon^{m})$ cannot be neglected as it represents an exponentially small contribution. This contrasts with that of Case $\text{A}$ (cf. Remark \ref{rem020103}(a)).
	\end{remark}
	
	We now proceed to derive the asymptotic series of $q^{\varepsilon}$ for the remaining points $(x,s)$ with $u(x,s)\neq{0}$. For such points, the function $u$ can be expressed as follows.
	\begin{lemma}\label{theo030104}
		In Cases $\text{B}_{I}$, $\text{B}_{II}$, $\text{B}_{III}$, and $\text{B}_{IV}$, we have
		\begin{equation}\label{eq030102}
			u(x,s)=\inf_{d\in\partial{D}}\,\inf_{s<t\leq{T}}u(x,s;d,t),
		\end{equation}
		for every $(x,s)\in{D}\times[0,T)$, where
		\begin{equation}\label{eq030103}
			\begin{aligned}
				u(x,s;d,t)\triangleq&\inf_{\gamma\in\mathcal{C}([s,T],\mathbb{R}),\;\gamma(s)=x,\;\gamma(t)=d}I(\gamma)\\
				=&\frac{a_1\sinh(a_1(T-s))\left(d-x^{0}_{x,s}(t)\right)^2}{2\sinh(a_1(T-t))\sinh(a_1(t-s))}.
			\end{aligned}
		\end{equation}
	\end{lemma}
	\begin{proof}
		Fix $d\in\partial{D}$ and $t\in(s,T]$. Let $\gamma$ denote a path contained in $\Gamma_{x,s}$ such that $\gamma(s)=x$ and $\gamma(t)=d$. The constrained optimization problem in (\ref{eq030103}) over such paths is a standard problem in the calculus of variations with fixed endpoints. The associated Hamilton equations are
		\begin{equation*}
			\frac{\mathrm{d}\gamma}{\mathrm{d}v}=\frac{\partial{H}}{\partial{\alpha}},\quad\frac{\mathrm{d}\alpha}{\partial{v}}=-\frac{\partial{H}}{\partial{\gamma}},\quad s<v<t,
		\end{equation*}
		with boundary conditions $\gamma(s)=x$, $\gamma(t)=d$, and Hamiltonian
		\begin{equation*}
			H(\gamma,\alpha,v)\triangleq{b}(\gamma,v)\alpha+\frac{1}{2}\alpha^2.
		\end{equation*}
		For each $(d,t)\in\partial{D}\times(s,T]$, there exists a unique smooth solution $\gamma_{x,s;d,t}\in\mathcal{C}^{\infty}([s,t],\mathbb{R})$, given by
		\begin{equation*}
			\begin{aligned}
				\gamma_{x,s;d,t}(v)=&\left(x+\frac{a_0}{a_1}\right)\frac{\sinh(a_1(t-v))}{\sinh(a_1(t-s))}\\
				+&\left(d+\frac{a_0}{a_1}\right)\frac{\sinh(a_1(v-s))}{\sinh(a_1(t-s))}-\frac{a_0}{a_1}, \quad v\in[s,t].
			\end{aligned}
		\end{equation*}
		Now let $\gamma:[s,T]\to\mathbb{R}$ be a path defined by
		\begin{equation*}
			\gamma(v)=\begin{cases}
				\gamma_{x,s;d,t}(v), & v\in[s,t]\\
				x^{0}_{d,t}(v),& v\in(t,T].
			\end{cases}
		\end{equation*}
		Then,
		\begin{equation*}
			\begin{aligned}
				u(x,s;d,t)=&\frac{1}{2}\int_{s}^{T}[\dot{\gamma}(v)-b(\gamma(v),v)]^2\mathrm{d}v\\
				=&\frac{1}{2}\int_{s}^{t}[\dot{\gamma}_{x,s;d,t}(v)-b(\gamma_{x,s;d,t}(v),v)]^2\mathrm{d}v,
			\end{aligned}
		\end{equation*}
		which, after evaluating the integral, yields the closed-form expression in (\ref{eq030103}). 
		
		Finally, by the definition of $u$, the minimum of $I(\gamma)$ over all paths starting at $x$ at time $s$ and reaching $\partial{D}$ before time $T$ equals the minimum of $u(x,s;d,t)$ over $(d,t)\in\partial{D}\times(s,T]$. Hence, Equation (\ref{eq030102}) follows. This completes the proof.
	\end{proof}
	\begin{remark}\label{rem030102}
		In Case $\text{B}_{I,II}$, we have
		\begin{equation*}
			u(x,s)=u\left(x,s;x^{0}_{x,s}(\tau^{0}_{x,s}),\tau^{0}_{x,s}\right)\equiv{0},\quad\forall\;(x,s)\in\Lambda_1\cup\Lambda_2.
		\end{equation*}
		This is consistent with Proposition \ref{theo030103}.
	\end{remark}
	\begin{remark}\label{rem030103}
		(a) For any $v\in(s,t)$, let $y=\gamma_{x,s;d,t}(v)$. One can verify that
		\begin{equation*}
			\gamma_{x,s;y,v}(\tau)=\gamma_{x,s;d,t}(\tau),\quad\forall \tau\in[s,v],
		\end{equation*}
		\begin{equation*}
			\gamma_{y,v;d,t}(\tau)=\gamma_{x,s;d,t}(\tau),\quad\forall \tau\in[v,t].
		\end{equation*}
		That is, every segment of $\gamma_{x,s;d,t}$ inherits the same functional form as the original path.
		
		(b) It is easy to show that
		\begin{equation*}
			u(x,s;d,t)\leq u(x,s;y,v)+u(y,v;d,t),\quad \forall (y,v)\in \mathbb{R}\times(s,t),
		\end{equation*}
		with equality if and only if $y=\gamma_{x,s;d,t}(v)$.
		
		(c) Note that for each $d\in\partial{D}$, 
		\begin{equation}\label{eq030104}
				u(x,s;d,t)\simeq\frac{a_1\left(d-x\right)^2}{2\sinh(a_1(t-s))}\to+\infty,\quad\text{as}\;t\to{s},
		\end{equation}
		\begin{equation}\label{eq030105}
			u(x,s;d,t)\simeq\frac{a_1\left(d-x_T\right)^2}{2\sinh(a_1(T-t))}\to+\infty,\quad\text{as}\;t\to{T}.
		\end{equation}
		The infimum of $u(x,s;d,t)$ is attained at one or more points in $\partial{D}\times(s,T)$. Consequently,
		\begin{equation}\label{eq030106}
			u(x,s)=\min_{d\in\partial{D}}\,\min_{s<t<T}u(x,s;d,t)
		\end{equation}
		
		(d) Moreover, by Corollary \ref{theo030102}, for $(x,s)\in{D}\times[0,T)$ in Cases $\text{B}_{I,I}$, $\text{B}_{II}$, $\text{B}_{III}$, and $\text{B}_{IV}$ (and for $(x,s)\in\Lambda_3$ in Case $\text{B}_{I,II}$),	we have
		\begin{equation*}
			u(x,s;d,t)>0,\quad\forall(d,t)\in\partial{D}\times(s,T),
		\end{equation*}
		which implies $u(x,s)>0$.
	\end{remark}
	
	For any $(x,s)$ with $u(x,s)>0$, let $(d_{x,s},\nu_{x,s})\in\partial{D}\times(s,T)$ be a point at which the infimum of $u(x,s;d,t)$ is attained. Regardless of whether this point is unique, we refer to $d_{x,s}$ as an optimal exit position and to $\nu_{x,s}$ as an optimal exit time. The associated optimal path is defined as
	\begin{equation*}
		\gamma_{x,s}(t)\triangleq\gamma_{x,s;d_{x,s},\nu_{x,s}}(t),\quad\forall\;t\in[s,\nu_{x,s}],
	\end{equation*} 
	which is clearly of class $\mathcal{C}^{\infty}$ on $[s,\nu_{x,s}]$. 
	\begin{lemma}\label{theo030105}
		For every $(x,s)$ with $u(x,s)>0$, the following properties hold:
		
		(a) $u(x,s)\to\infty$ as $s\to{T}$, uniformly for $x$ in every compact subset of $D$; 
		
		(b) For each $T'\in[0,T)$, $u(x,s)$ is bounded on ${D}\times[0,T']$;
		
		(c) $\nu_{x,s}\to{T}$ as $s\to{T}$, uniformly for $x$ in $D$; 
		
		(d) For each $T'\in[0,T)$, there exists $\delta>0$ such that $\nu_{x,s}\in(s,T-\delta)$ for all $(x,s)\in{D}\times[0,T']$; 
		
		(e) For every compact $K\subset{D}\times[0,T)$, there exists $\delta>0$ such that $\nu_{x,s}\in(s+\delta,T-\delta)$ for all $(x,s)\in{K}$.
		
		(f) $u(x,s)\to{0}$ as $\text{dist}(x,\partial{D})\to{0}$, uniformly for $s$ in every compact subset of $[0,T)$;
		
		(g) $\nu_{x,s}\to{s}$ as $\text{dist}(x,\partial{D})\to{0}$, uniformly for $s$ in every compact subset of $[0,T)$;
		
		(h) For each $T'\in[0,T)$, $u(x,s)$ is Lipschitz continuous on $\bar{D}\times[0,T']$.
	\end{lemma}
	\begin{proof}
		(a) Let $E\subset{D}$ be compact. By Corollary \ref{theo030102}, $x^{0}_{x,s}(t)\in{D}$ for all $t\in[s,T]$ and all $x\in{E}$ as $s\to{T}$. Hence, there exists a constant $C$ such that
		\begin{equation*}
			\left\vert d-x^{0}_{x,s}(t)\right\vert>C>0,\quad\forall\;t\in[s,T],\;\forall d\in\partial{D},\;\forall\;x\in{E}.
		\end{equation*}
		Then for any $d\in\partial{D}$ and $t\in(s,T)$,
		\begin{equation*}
			\begin{aligned}
				u(x,s;d,t)&\simeq\frac{(T-s)(d-x^{0}_{x,s}(t))^2}{2(T-t)(t-s)}\\
				&\geq\frac{2C^2}{T-s}\to+\infty,\quad\forall\;x\in{E},\;\text{as}\;s\to{T},
			\end{aligned}
		\end{equation*}
		which proves (a).
		
		(b) Fix $T'\in[0,T)$. Let $(x,s)\in{D}\times[0,T']$ be a point so that $x^{0}_{x,s}(t)\in{D}$ for all $t\in[s,T]$. Then for any $d\in\partial{D}$,
		\begin{equation*}
			\begin{aligned}
				u(x,s)\leq& u(x,s;d,(T+s)/2)\\
				=&\frac{a_1(d-x^{0}_{x,s}((T+s)/2))^2}{\tanh(a_1(T-s)/2)}\\
				\leq&\frac{a_1(d_2-d_1)^2}{\tanh(a_1(T-T')/2)}<\infty,
			\end{aligned}
		\end{equation*}
		which proves (b).
		
		(c) This follows directly from the definition of $\nu_{x,s}$, i.e., $\nu_{x,s}\in(s,T)$ for all $x$ in $D$.
		
		(d) Set $M\triangleq\max_{(x,s)\in{D}\times[0,T']}u(x,s)+1$. By property (\ref{eq030105}), there exists $\delta>0$ such that
		\begin{equation*}
			\begin{aligned}
				&u(x,s;d,t)>M,\\
				&\forall(x,s)\in{D}\times[0,T'],\;\forall(d,t)\in\partial{D}\times[T-\delta,T).
			\end{aligned}
		\end{equation*}
		Since $u(x,s)=u(x,s;d_{x,s},\nu_{x,s})<M$, the pair $(d_{x,s},\nu_{x,s})$ cannot belong to $\partial{D}\times[T-\delta,T)$. Therefore, $\nu_{x,s}\in(s,T-\delta)$.
		
		(e) Let $K$ be a compact subset of ${D}\times[0,T)$ and set $M\triangleq\max_{(x,s)\in{K}}u(x,s)+1$. By property (\ref{eq030104}), we can choose $\delta>0$ so that
		\begin{equation*}
			u(x,s;d,t)>M,\quad\forall{(x,s)}\in{K},\;\forall(d,t)\in\partial{D}\times(s,s+\delta].
		\end{equation*}
		Therefore, $(d_{x,s},\nu_{x,s})\notin\partial{D}\times(s,s+\delta]$. Combined with (d), we conclude that $\nu_{x,s}\in(s+\delta,T-\delta)$ for all $(x,s)\in{K}$. 
		
		(f) As $\text{dist}(x,\partial{D})\to{0}$, pick $d\in\partial{D}$ with $\vert x-d\vert\to{0}$. Then
		\begin{equation*}
			\begin{aligned}
				u(x,s)\leq &u(x,s;d,s+\vert x-d\vert)\\
				\simeq&[a_1(d+a_0/a_1)\cosh(a_1(T-s))+\text{sign}(d-x)\\
				&\times\sinh(a_1(T-s))-a_1(x_T+a_0/a_1)]^2\\
				&/[2\sinh^2(a_1(T-s))]\times\vert x-d\vert+o(\vert x-d\vert),\\
			\end{aligned}
		\end{equation*}
		which tends to zero as $\vert x-d\vert\to{0}$, uniformly for $s$ in every compact subset of $[0,T)$. This proves (f).
		
		(g) Fix $T'\in[0,T)$ and $\delta\in(0,T-T')$. Let $(x,s)\in{D}\times[0,T']$ denote a point so that $x^{0}_{x,s}(t)\in{D}$ for all $t\in[s,T]$. Since $\text{dist}(x,\partial{D})\to{0}$, we choose $d\in\partial{D}$ such that $\vert x-d\vert\to{0}$. By the continuous dependence of $x^{0}_{x,s}$ on the initial point $x$, we know that $x^{0}_{d,s}(s)=d\in\partial{D}$ and $x^{0}_{d,s}(t)\in{D}$ for all $t\in(s,T]$. Define
		\begin{equation*}
			C_1(\delta)\triangleq\min_{t\in[s+\delta,T]}\vert d-x^{0}_{d,s}(t)\vert>0.
		\end{equation*}
		Choose $\eta_1>0$ sufficiently small so that for all $\vert x-d\vert<\eta_1$,
		\begin{equation*}
			\max_{t\in[s+\delta,T]}\vert x^{0}_{x,s}(t)-x^{0}_{d,s}(t)\vert<C_1(\delta)/2.
		\end{equation*}
		Then, for such $x$,
		\begin{equation*}
			\min_{t\in[s+\delta,T]}\vert d-x^{0}_{x,s}(t)\vert>C_1(\delta)/2.
		\end{equation*}
		Define
		\begin{equation*}
			\begin{aligned}
				C_2=&\min_{t\in(s,T)}\frac{a_1\sinh(a_1(T-s))}{2\sinh(a_1(T-t))\sinh(a_1(t-s))}\\
				=&\frac{a_1}{\tanh(a_1(T-s)/2)}.
			\end{aligned}
		\end{equation*}
		Consequently, for all $t\in[s+\delta,T)$ and all $\vert x-d\vert<\eta_1$,
		\begin{equation*}
			u(x,s;d,t)>\frac{C_1^2(\delta)C_2}{4}.
		\end{equation*}
		By (f), $u(x,s)\to{0}$ as $\vert x-d\vert\to{0}$. Hence, there exists $\eta_2\in(0,\eta_1)$ such that $u(x,s)<{C_1^2(\delta)C_2}/{4}$ whenever $\vert x-d\vert<\eta_2$. Since the cost $u(x,s;d,t)$ for exiting at any time $t\geq s+\delta$ is bounded below by this positive constant, the optimal exit time $\nu_{x,s}$, which achieves the smaller cost $u(x,s)$, must satisfy $\nu_{x,s}<s+\delta$ for all $\vert x-d\vert<\eta_2$ and all $s\in[0,T']$. This proves (g).
		
		(h) Fix $T'\in[0,T)$. By (d), there exists $\delta>0$ such that $\nu_{x,s}\in(s,T-\delta)$ for all $(x,s)\in{D}\times[0,T']$. Choose a constant $\eta$ with $0<\eta<\delta/2$. Take any two points $(x_1,s_1),\;(x_2,s_2)\in{D}\times[0,T']$ satisfying $\vert x_1-x_2\vert+\vert s_1-s_2\vert<\eta$. Then
		\begin{equation*}
			\begin{aligned}
				\vert u(x_1,s_1)-u(x_2,s_2)\vert\leq&\vert u(x_1,s_1)-u(x_1,s_2)\vert\\
				&+\vert u(x_1,s_2)-u(x_2,s_2)\vert.
			\end{aligned}
		\end{equation*}
		We estimate each term separately.
		
	    \textit{Estimate of the first term:} Let $\gamma_{x_1,s_2}$ be an optimal path for $(x_1,s_2)$. Define a shifted path $\gamma^{(1)}$ by
		\begin{equation*}
			\gamma^{(1)}(t)\triangleq\gamma_{x_1,s_2}(t-s_1+s_2),\quad t\in[s_1,\nu_{x_1,s_2}+s_1-s_2].
		\end{equation*}
		Then
		\begin{equation*}
			\begin{array}{c}
				\gamma^{(1)}(s_1)=x_1,\quad\gamma^{(1)}(\nu_{x_1,s_2}+s_1-s_2)=d_{x_1,s_2}\in\partial{D},\\
				\nu_{x_1,s_2}+s_1-s_2\in(s_1,T-\delta/2).
			\end{array}
		\end{equation*}
		That is, $(\gamma^{(1)}(t),t)$ is a trajectory connecting $(x_1,s_1)$ with the boundary $\partial{D}\times[0,T)$. By the definition of $u$, we have
		\begin{equation*}
			\begin{aligned}
				u(x_1,s_1)\leq&\frac{1}{2}\int_{s_1}^{\nu_{x_1,s_2}+s_1-s_2}[\dot{\gamma}^{(1)}(t)-b(\gamma^{(1)}(t),t)]^2\mathrm{d}t\\
				=&\frac{1}{2}\int_{s_2}^{\nu_{x_1,s_2}}[\dot{\gamma}_{x_1,s_2}(v)-b(\gamma_{x_1,s_2}(v),v+s_1-s_2)]^2\mathrm{d}v.
			\end{aligned}
		\end{equation*}
		Consequently,
		\begin{equation*}
			u(x_1,s_1)-u(x_1,s_2)\leq C_1\vert s_1-s_2\vert,
		\end{equation*}
		where
		\begin{equation*}
			C_1\triangleq L\left(\sqrt{2T\max_{(x,s)\in{D}\times[0,T']}u(x,s)}+LT\eta/2\right),
		\end{equation*}
		and $L<\infty$ is a Lipschitz constant of $b(x,s)$ on the compact set $\bar{D}\times[0,T-\delta/2]$. Reversing the roles of $s_1$ and $s_2$ yields the same bound. Therefore,
		\begin{equation*}
			\vert u(x_1,s_1)-u(x_1,s_2)\vert\leq C_1\vert s_1-s_2\vert.
		\end{equation*}
		
		\textit{Estimate of the second term:} Set $\Delta\triangleq x_2-x_1$ and consider the translated path
		\begin{equation*}
			\gamma^{(2)}(t)\triangleq\gamma_{x_1,s_2}(t)+\Delta,\quad t\in[s_2,\nu_{x_1,s_2}].
		\end{equation*}
		If $d_{x_1,s_2}+\Delta\notin{D}$, then $\gamma^{(2)}$ exits $D$ at some time before $\nu_{x_1,s_2}$. Otherwise, we extend it after $\nu_{x_1,s_2}$ by an optimal continuation from the point $(d_{x_1,s_2}+\Delta,\nu_{x_1,s_2})$ to the boundary $\partial{D}\times[0,T)$. In either case, using the result of (f) and the Lipschitz continuity of $b$, one can show that
		\begin{equation*}
			u(x_2,s_2)-u(x_1,s_2)\leq C_2\vert\Delta\vert,
		\end{equation*}
		where $C_2<\infty$ is a constant depending only on $T'$, $\eta$, and the Lipschitz constant of $b$. Exchanging $x_1$ and $x_2$ gives the symmetric bound, so
		\begin{equation*}
			\vert u(x_2,s_2)-u(x_1,s_2)\vert\leq C_2\vert x_2-x_1\vert.
		\end{equation*}
		
		Combining the two estimates, we obtain a constant $C\triangleq\max\{C_1,C_2\}$ such that
		\begin{equation*}
			\vert u(x_1,s_1)-u(x_2,s_2)\vert\leq C \left(\vert s_1-s_2\vert+ \vert x_2-x_1\vert\right),
		\end{equation*}
		for all $(x_1,s_1),\;(x_2,s_2)\in{D}\times[0,T']$ with $\vert x_1-x_2\vert+\vert s_1-s_2\vert<\eta$. Since $\eta$ depends only on $T'$ and $\delta$, this proves that $u$ is locally Lipschitz on ${D}\times[0,T']$. Because $u$ extends continuously to $\bar{D}\times[0,T']$ with $u=0$ on $\partial{D}\times[0,T']$ and $\bar{D}\times[0,T']$ is compact, we conclude that $u$ is globally Lipschitz on $\bar{D}\times[0,T']$. This completes the proof.
	\end{proof}	
	The notion of a strongly regular point is defined as follows.
	\begin{definition}\label{def030101}
		A point $(x,s)\in{D}\times[0,T)$ with $u(x,s)>0$ is said to be strongly regular if the following conditions hold:
		\begin{itemize}
			\item[(a)] The minimum of $u(x,s;d,t)$ over $(d,t)\in\partial{D}\times(s,T)$ is attained at a unique point $(d_{x,s},\nu_{x,s})\in\partial{D}\times(s,T)$. 
			\item[(b)] The function $t\mapsto u(x,s;d_{x,s},t)$ has a non-degenerate minimum at $t=\nu_{x,s}$, i.e., its second derivative with respect to $t$ is positive at $t=\nu_{x,s}$.
		\end{itemize}
	\end{definition}
	\begin{remark}\label{rem030104}
		
		Define
		\begin{equation*}
				f(x,s;d,t)\triangleq a_1\sinh(a_1(T-s))\left(d-x^{0}_{x,s}(t)\right),
		\end{equation*}
		and
		\begin{equation*}
				\begin{aligned}
						g(x,s;d,t)\triangleq &a_1\bigg[\left(d+\frac{a_0}{a_1}\right)\sinh(a_1(2t-T-s))+\left(x+\frac{a_0}{a_1}\right)\\
						\times&\sinh(a_1(T-t))-\left(x_T+\frac{a_0}{a_1}\right)\sinh(a_1(t-s))\bigg].
					\end{aligned}
			\end{equation*}
		One may verify that
		\begin{equation*}
				\frac{\partial u(x,s;d,t)}{\partial{t}}=\frac{f(x,s;d,t)g(x,s;d,t)}{2\sinh^2(a_1(t-s))\sinh^2(a_1(T-t))}.
		\end{equation*}
		For any $(x,s)$ with $u(x,s)>0$, we have
		\begin{equation*}
			f(x,s;d,t)\neq{0},\quad\forall \;(d,t)\in\partial{D}\times[s,T].
		\end{equation*}
		Therefore, regardless of whether the minimizer $(d_{x,s},\nu_{x,s})$ is unique, the optimality of $(d_{x,s},\nu_{x,s})$ implies
		\begin{equation*}
			g(x,s;d_{x,s},t)\vert_{t=\nu_{x,s}}=0.
		\end{equation*}
		Moreover, when condition (a) of Definition \ref{def030101} is valid, condition (b) is equivalent to requiring
		\begin{equation*}
			\left\{\begin{aligned}
				&g(x,s;d_{x,s},t)\vert_{t=\nu_{x,s}}=0,\\
				&\frac{\partial g(x,s;d_{x,s},t)}{\partial t}\Big\vert_{t=\nu_{x,s}}\neq{0}.
			\end{aligned}\right.
		\end{equation*}
	\end{remark}
	
	From these facts we obtain the following lemma.
	\begin{lemma}\label{theo030106}
		For any $(x,s)$ with $u(x,s)>0$, the following properties hold:
		
		(a) Regardless of whether the minimum of $u(x,s;d,t)$ is uniquely attained or not, the optimal path $\gamma_{x,s}(t)$ remains in ${D}$ for all $t\in[s,\nu_{x,s})$. Moreover, $\dot{\gamma}_{x,s}(\nu_{x,s})\neq{0}$. That is, $(\gamma_{x,s}(t),t)$ crosses $\partial{D}\times[0,T)$ non-tangentially.
		
		(b) Regardless of whether the minimum of $u(x,s;d,t)$ is uniquely attained or not, the pair $(d_{x,s},\nu_{x,s})$ continues to serve as the optimal exit position and optimal exit time for every point $(y,v)=(\gamma_{x,s}(v),v)$ with $v\in(s,\nu_{x,s})$. Furthermore, for such $(y,v)$, the minimum of $u(y,v;d,t)$ is attained uniquely at $(d_{x,s},\nu_{x,s})$, and consequently, every such point $(y,v)$ is strongly regular. (In other words, the optimal exit position and optimal exit time are invariant along the graph $(\gamma_{x,s}(t),t)$, which parallels Remark \ref{rem020102}.)
		
		(c) Among points where the minimum of $u(x,s;d,t)$ is uniquely attained, at most one may fail to satisfy condition (b) of Definition \ref{def030101} (i.e., be degenerate).
	\end{lemma}
	\begin{proof}
		(a) Assume, for contradiction, that there exists $t_0\in[s,\nu_{x,s})$ with $\gamma_{x,s}(t_0)\in\partial{D}$. Then
		\begin{equation*}
			u(x,s)\leq u(x,s;\gamma_{x,s}(t_0),t_0)<u(x,s;d_{x,s},\nu_{x,s}),
		\end{equation*}
		which contradicts the definition of $(d_{x,s},\nu_{x,s})$.
		
		Now suppose $\dot{\gamma}_{x,s}(\nu_{x,s})={0}$. This condition implies
		\begin{equation*}
			\left(d_{x,s}+{a_0}/{a_1}\right)\cosh(a_1(\nu_{x,s}-s))=x+a_0/a_1,
		\end{equation*}
		and consequently,
		\begin{equation*}
			\begin{aligned}
				g(x,s;d_{x,s},\nu_{x,s})=&a_1\sinh(a_1(\nu_{x,s}-s))[\left(d_{x,s}+{a_0}/{a_1}\right)\\
				&\times\cosh(a_1(T-\nu_{x,s}))-\left(x_T+a_0/a_1\right)].
			\end{aligned}
		\end{equation*}
		One may verify that this situation can occur only when $d_{x,s}=d_2$ in Case $\text{B}_{I}$. In Subcase $\text{B}_{I,I}$, the inequality $\nu_{x,s}>0$ for all $(x,s)\in{D}\times[0,T)$ yields
		\begin{equation*}
			\begin{aligned}
				g(x,s;d_{x,s},\nu_{x,s})>&a_1\sinh(a_1(\nu_{x,s}-s))[\left(d_{x,s}+{a_0}/{a_1}\right)\\
				&\times\cosh(a_1T)-\left(x_T+a_0/a_1\right)]\geq0.
			\end{aligned}
		\end{equation*}
		In Subcase $\text{B}_{I,II}$, however, we have $\nu_{x,s}>t_4$ for all $(x,s)\in\Lambda_3$. Therefore,
		\begin{equation*}
			\begin{aligned}
				g(x,s;d_{x,s},\nu_{x,s})>&a_1\sinh(a_1(\nu_{x,s}-s))[\left(d_{x,s}+{a_0}/{a_1}\right)\\
				&\times\cosh(a_1(T-t_4))-\left(x_T+a_0/a_1\right)]=0.
			\end{aligned}
		\end{equation*}
		In both subcases, we obtain $g(x,s;d_{x,s},\nu_{x,s})>0$, contradicting the definition of $(d_{x,s},\nu_{x,s})$. This proves (a).
		
		(b) Suppose $(d_{x,s},\nu_{x,s})$ does not minimize $u(y,v;d,t)$. Let $(d_{y,v},\nu_{y,v})$ be an optimal exit pair for $(y,v)$. Construct a new curve $(\gamma(t),t)$ connecting $(x,s)$ with $(d_{y,v},\nu_{y,v})\in\partial{D}\times[0,T)$ by
		\begin{equation*}
			\gamma(t)=\begin{cases}
				\gamma_{x,s}(t), & t\in[s,v),\\
				\gamma_{y,v;d_{y,v},\nu_{y,v}}(t), &t\in[v,\nu_{y,v}].
			\end{cases}
		\end{equation*}
		Then
		\begin{equation*}
			\begin{aligned}
				I(\gamma)&=u(x,s;y,v)+u(y,v;d_{y,v},\nu_{y,v})\\
				&<u(x,s;y,v)+u(y,v;d_{x,s},\nu_{x,s})\\
				&=u(x,s;d_{x,s},\nu_{x,s})\\
				&=u(x,s),
			\end{aligned}
		\end{equation*}
		which contradicts the optimality of $\gamma_{x,s}$.
		
		Now suppose the minimum of $u(y,v;d,t)$ is not uniquely attained. So, besides $(d_{x,s},\nu_{x,s})$, there exists another $(d^{*},\nu^{*})\in\partial{D}\times[0,T)$ such that
		\begin{equation*}
			u(y,v)=u(y,v;d_{x,s},\nu_{x,s})=u(y,v;d^{*},\nu^{*}).
		\end{equation*}
		Define
		\begin{equation*}
			\gamma(t)=\begin{cases}
				\gamma_{x,s}(t), & t\in[s,v),\\
				\gamma_{y,v;d^{*},\nu^{*}}(t), &t\in[v,\nu^{*}].
			\end{cases}
		\end{equation*}
		Then
		\begin{equation*}
			\begin{aligned}
				I(\gamma)&=u(x,s;y,v)+u(y,v;d^{*},\nu^{*})\\
				&=u(x,s;y,v)+u(y,v;d_{x,s},\nu_{x,s})\\
				&=u(x,s;d_{x,s},\nu_{x,s})\\
				&=u(x,s).
			\end{aligned}
		\end{equation*}
		This implies $\gamma(t)$ is also optimal for $(x,s)$. Because $(d^{*},\nu^{*})\neq(d_{x,s},\nu_{x,s})$, we have $\dot{\gamma}_{x,s}(v)\neq\dot{\gamma}_{y,v;d^{*},\nu^{*}}(v)$, and thus $\gamma(t)$ fails to be differentiable at $t=v$. This contradicts the $\mathcal{C}^{\infty}$-smoothness of optimal paths.
		
		Finally, suppose $(y,v)$ is not strongly regular. Then
		\begin{equation*}
			\left\{\begin{aligned}
				&\frac{\partial u(y,v;d_{x,s},t)}{\partial t}\Big\vert_{t=\nu_{x,s}}=0,\\
				&\frac{\partial^2 u(y,v;d_{x,s},t)}{\partial t^2}\Big\vert_{t=\nu_{x,s}}=0.
			\end{aligned}\right.
		\end{equation*}
		For any $t$ near $\nu_{x,s}$, the additivity of the action (cf. Remark \ref{rem030103}(b)) gives
		\begin{equation*}
			u(x,s;d_{x,s},t)=u(x,s;\gamma_{x,s;d_{x,s},t}(v),v)+u(\gamma_{x,s;d_{x,s},t}(v),v;d_{x,s},t).
		\end{equation*}
		Differentiating this identity once and twice with respect to $t$ and evaluating at $t=\nu_{x,s}$ yields
		\begin{equation*}
			\frac{\partial u(x,s;d_{x,s},t)}{\partial t}\Big\vert_{t=\nu_{x,s}}=0,
		\end{equation*}
		and
		\begin{equation*}
			\begin{aligned}
				\frac{\partial^2 u(x,s;d_{x,s},t)}{\partial t^2}\Big\vert_{t=\nu_{x,s}}&=\frac{-a_1\sinh(a_1(v-s))(\dot{\gamma}_{x,s}(\nu_{x,s}))^2}{\sinh(a_1(\nu_{x,s}-v))\sinh(a_1(\nu_{x,s}-s))}\\
				&<0.
			\end{aligned}
		\end{equation*}
		This contradicts that $\nu_{x,s}$ is a local minimizer of $u(x,s;d_{x,s},t)$.
		
		(c) Suppose the minimum of $u(x,s;d,t)$ is uniquely attained at $(d^{*},\nu^{*})=(d_{x,s},\nu_{x,s})$, but condition (b) of Definition \ref{def030101} fails. Since $u(x,s;d^{*},t)$ is analytic in $t$, the optimality of $\nu^{*}$ implies the existence of an integer $m>1$ such that
		\begin{equation*}
			\left\{\begin{aligned}
				&g(x,s;d^{*},t)\vert_{t=\nu^{*}}=0, \\
				&\frac{\partial g(x,s;d^{*},t)}{\partial{t}}\Big\vert_{t=\nu^{*}}=0,\\
				&\frac{\partial^{k} g(x,s;d^{*},t)}{\partial{t}^k}\Big\vert_{t=\nu^{*}}=0,\quad k=2,\cdots,2m-2,\\
				&\frac{\partial^{2m-1} g(x,s;d^{*},t)}{\partial{t}^{2m-1}}\Big\vert_{t=\nu^{*}}\neq{0}.
			\end{aligned}\right.
		\end{equation*}
		Observe that for any $m>1$,
		\begin{equation*}
			\begin{aligned}
				\frac{\partial^{2m-2}g(x,s;d^{*},t)}{\partial{t}^{2m-2}}=&a_1^{2m-1}(2^{2m-2}-1)(d^{*}+a_0/a_1)\\
				\times\sinh&(a_1(2t-T-s))+a_1^{2m-2}g(x,s;d^{*},t),\\
				\frac{\partial^{2m-1}g(x,s;d^{*},t)}{\partial{t}^{2m-1}}=&a_1^{2m}(2^{2m-1}-2)(d^{*}+a_0/a_1)\\
				\times\cosh&(a_1(2t-T-s))+a_1^{2m-2}\frac{\partial g(x,s;d^{*},t)}{\partial {t}}.\\
			\end{aligned}
		\end{equation*}
		Consequently,
		\begin{equation*}
			d^{*}+a_0/a_1\neq{0},\quad\nu^{*}=(T+s)/2, \quad m=2.
		\end{equation*}
		Substituting them into the equations $g=0$ and $\frac{\partial g}{\partial t}=0$ gives
		\begin{equation*}
			x=x_T,\quad d^{*}+a_0/a_1=(x_T+a_0/a_1)\cosh(a_1(T-s)/2).
		\end{equation*}
		One may verify that:
		\begin{itemize}
			\item This situation occurs only when $d^{*}=d_1$ in Cases $\text{B}_{I}$, $\text{B}_{II}$, and $\text{B}_{III}$;
			\item The above equations admit at most one solution $(x,s)\in{D}\times(-\infty,T)$. Moreover, the solution $(x,s)$ lies in $D\times[0,T)$ only if $x_T$ is sufficiently close to $d_1$ or $T$ is sufficiently large;
			\item A point $(x,s)$ where the minimum of $u(x,s;d,t)$ is uniquely attained yet condition (b) of Definition \ref{def030101} fails can exist only when it solves the above equations, belongs to $D\times[0,T)$, and satisfies
			\begin{equation*}
				\min_{t\in(s,T)}u(x,s;d_1,t)<\min_{t\in(s,T)}u(x,s;d_2,t).
			\end{equation*}
		\end{itemize} 
		This completes the proof.
	\end{proof}
	
	The following proposition allows us to determine for which points $(x,s)$ the function $u(x,s)$ is $\mathcal{C}^{\infty}$.
	\begin{proposition}\label{theo030107}
		Let $(x,s)\in{D}\times[0,T)$ be a strongly regular point. Then there exists an open neighborhood $U$ of $(x,s)$ such that every point in $U$ is strongly regular. Moreover, the mappings $(y,v)\to(d_{y,v},\nu_{y,v})$ and $(y,v)\to u(y,v)$ are $\mathcal{C}^{\infty}$ on $U$, and
		\begin{equation}\label{eq030107}
			\frac{\partial u(y,v)}{\partial{y}}=-\frac{a_1(d_{y,v}-x^{0}_{y,v}(\nu_{y,v}))}{\sinh(a_1(\nu_{y,v}-v))},
		\end{equation}
		\begin{equation}\label{eq030108}
			\frac{\partial^2 u(y,v)}{\partial{y}^2}=\frac{a_1^3(d_{y,v}+a_0/a_1)[1-\cosh(2a_1(T-\nu_{y,v}))]}{\sinh(a_1(T-v))\frac{\partial g(y,v;d_{y,v},t)}{\partial t}\big\vert_{t=\nu_{y,v}}}.
		\end{equation}
	\end{proposition}
	\begin{proof}
		We first show that in a neighborhood $U_1$ of $(x,s)$ the optimal exit position $d_{y,v}$ is constant. Suppose this is not true. Then for every $1/n$-neighborhood $V_{1/n}$ of $(x,s)$, there exist points $(y^{(1)}_n,v^{(1)}_n)$ and $(y^{(2)}_n,v^{(2)}_n)$ in $V_{1/n}$ such that $d_{y^{(1)}_n,v^{(1)}_n}=d_1$ and $d_{y^{(2)}_n,v^{(2)}_n}=d_2$. By definition, this implies the existence of parameters $t^{(1)}_n$ and $t^{(2)}_n$ such that
		\begin{equation*}
			\begin{aligned}
				u(y^{(1)}_n,v^{(1)}_n)=u(y^{(1)}_n,v^{(1)}_n;d_1,t^{(1)}_n),\\
				u(y^{(2)}_n,v^{(2)}_n)=u(y^{(2)}_n,v^{(2)}_n;d_2,t^{(2)}_n).
			\end{aligned}
		\end{equation*}
		Because $u$ is Lipschitz continuous (Lemma \ref{theo030105}(h)), we have
		\begin{equation*}
			\begin{aligned}
				\vert u(x,s)-u(y^{(1)}_n,v^{(1)}_n;d_1,t^{(1)}_n)\vert=O(1/n),\\
				\vert u(x,s)-u(y^{(2)}_n,v^{(2)}_n;d_2,t^{(2)}_n)\vert=O(1/n).
			\end{aligned}
		\end{equation*}
		For sufficiently large $n$, Lemma \ref{theo030105}(e) guarantees that both $t^{(1)}_n$ and $t^{(2)}_n$ lie in a compact interval $[s+\delta,T-\delta]$. Hence, by passing to subsequences if necessary, we may assume $t^{(1)}_n\to t^{(1)}$ and $t^{(2)}_n\to t^{(2)}$ as $n\to\infty$, with $t^{(1)},\;t^{(2)}\in(s,T)$. The continuity of $u(x,s;d,t)$ then gives
		\begin{equation*}
			u(x,s)=u(x,s;d_1,t^{(1)})=u(x,s;d_2,t^{(2)}).
		\end{equation*}
		So, we obtain two distinct optimal trajectories $(\gamma_{x,s;d_1,t^{(1)}}(t),t)$ and $(\gamma_{x,s;d_2,t^{(2)}}(t),t)$ emanating from $(x,s)$. This contradicts the strong regularity of $(x,s)$.
		
		Now consider the function $g(y,v;d_{y,v},t)$ for $(y,v)\in U_1$ and $t\in(0,T)$. Since we have just shown that $d_{y,v}=d_{x,s}$ for all $(y,v)\in U_1$, it follows that $g(y,v;d_{y,v},t)=g(y,v;d_{x,s},t)$ for all $(y,v)\in U_1$ and all $t\in(0,T)$. Because $t=\nu_{x,s}$ is a local minimizer of $u(x,s;d_{x,s},t)$, we have
		\begin{equation*}
			\left\{\begin{aligned}
				&g(x,s;d_{x,s},t)\vert_{t=\nu_{x,s}}=0,\\
				&\frac{\partial g(x,s;d_{x,s},t)}{\partial t}\Big\vert_{t=\nu_{x,s}}\neq{0}.
			\end{aligned}\right.
		\end{equation*}
		By the implicit function theorem and the continuity of $\frac{\partial g}{\partial t}$, there exist neighborhoods $U_2\subset U_1$ of $(x,s)$ and $W_1\subset(0,T)$ of $\nu_{x,s}$ such that for every $(y,v)\in U_2$ there is a unique $\nu_{y,v}\in W_1$ satisfying
		\begin{equation*}
			\left\{\begin{aligned}
				&g(y,v;d_{y,v},t)\vert_{t=\nu_{y,v}}=g(y,v;d_{x,s},t)\vert_{t=\nu_{y,v}}=0,\\
				&\frac{\partial g(y,v;d_{y,v},t)}{\partial t}\Big\vert_{t=\nu_{y,v}}=\frac{\partial g(y,v;d_{x,s},t)}{\partial t}\Big\vert_{t=\nu_{y,v}}\neq{0}.
			\end{aligned}\right.
		\end{equation*}
		The $\mathcal{C}^{\infty}$-smoothness of $g(y,v;d_{x,s},t)$ indicates that the mapping $(y,v)\mapsto \nu_{y,v}$ is $\mathcal{C}^{\infty}$ on $U_2$, and
		\begin{equation*}
			\frac{\partial \nu_{y,v}}{\partial y}=-{\frac{\partial g(y,v;d_{y,v},t)}{\partial y}\Big\vert_{t=\nu_{y,v}}}\Big/{\frac{\partial g(y,v;d_{y,v},t)}{\partial t}\Big\vert_{t=\nu_{y,v}}}.
		\end{equation*}
		
		Since $t=\nu_{x,s}$ is the unique minimizer of $u(x,s;d_{x,s},t)$, for any neighborhood $W_2$ of $\nu_{x,s}$, there exists a neighborhood $U_3\subset U_1$ of $(x,s)$ so that for every $(y,v)\in U_3$ the minimum of $u(y,v;d_{y,v},t)$ is attained within $W_2$. Hence, if $(y,v)\in U_2\cap U_3$, the unique solution of $g(y,v;d_{y,v},t)=0$ in $W_1\cap W_2$ coincides with the minimizer of $u(y,v;d_{y,v},t)$. Let $U\triangleq U_2\cap U_3$. Consequently, every point $(y,v)\in U$ is strongly regular, and its optimal exit time is precisely $\nu_{y,v}$.
		
		Finally, the mapping $(y,v)\mapsto u(y,v;d_{y,v},\nu_{y,v})$ is $\mathcal{C}^{\infty}$ because $(y,v)\mapsto(d_{y,v},\nu_{y,v})$ and $(y,v,d,t)\mapsto u(y,v;d,t)$ are $\mathcal{C}^{\infty}$. Differentiating the identity $u(y,v)=u(y,v;d_{y,v},\nu_{y,v})$ once and twice with respect to $y$ then yields formulas (\ref{eq030107}) and (\ref{eq030108}). This completes the proof.
	\end{proof}
	\begin{corollary}\label{theo030108}
		Let $(x,s)\in{D}\times[0,T)$ be a strongly regular point. Then the corresponding optimal path $\gamma_{x,s}(t)$ satisfy Equation (\ref{eqA10106}).
	\end{corollary}
	\begin{proof}
		Let $\gamma=\gamma_{x,s}=\gamma_{x,s;d_{x,s},\nu_{x,s}}$ denote the unique optimal path. Since $\gamma$ satisfies the associated constrained Hamilton equations, we have
		\begin{equation*}
			\dot{\gamma}(t)=b(\gamma(t),t)+\alpha(t).
		\end{equation*}
		By Lemma \ref{theo030106}(b), every point $(\gamma(t),t)$ with $t\in[s,\nu_{x,s})$ is strongly regular, so $d_{\gamma(t),t}=d_{x,s}$ and $\nu_{\gamma(t),t}=\nu_{x,s}$. Consequently,
		\begin{equation*}
			\begin{aligned}
				\alpha(t)=&\dot{\gamma}(t)-b(\gamma(t),t)\\
				=&\frac{a_1(d_{x,s}-x^{0}_{\gamma(t),t}(\nu_{x,s}))}{\sinh(a_1(\nu_{x,s}-t))}\\
				=&\frac{a_1(d_{\gamma(t),t}-x^{0}_{\gamma(t),t}(\nu_{\gamma(t),t}))}{\sinh(a_1(\nu_{\gamma(t),t}-t))}\\
				=&-\frac{\partial u(y,v)}{\partial y}\Big\vert_{(y,v)=(\gamma(t),t)}.
			\end{aligned}
		\end{equation*}
		Substituting this expression for $\alpha(t)$ back into the Hamilton equations then yields
		\begin{equation*}
			\dot{\gamma}(t)=b(\gamma(t),t)-\frac{\partial u}{\partial x}(\gamma(t),t),
		\end{equation*}
		which is exactly Equation (\ref{eqA10106}).
	\end{proof}
	\begin{remark}\label{rem030105}
		Suppose $(x,s)$ is not strongly regular. Then either the minimum of $u(x,s;d,t)$ is not uniquely attained, or it is uniquely attained but condition (b) of Definition \ref{def030101} fails. We conclude:
		
		(a) If the minimum of $u(x,s;d,t)$ is not uniquely attained, then $u(x,s)$ is non-differentiable at $(x,s)$. 
		
		Indeed, assume the minimum of $u(x,s;d,t)$ is attained at two distinct points $(d^{(1)},t^{(1)})$ and $(d^{(2)},t^{(2)})$, giving two distinct optimal paths $\gamma^{(1)}=\gamma_{x,s;d^{(1)},t^{(1)}}$ and $\gamma^{(2)}=\gamma_{x,s;d^{(2)},t^{(2)}}$. By Lemma \ref{theo030106}(b), for all $t\in(s,s+\delta)$ (with some $\delta>0$), the points $(\gamma^{(1)}(t),t)$ and $(\gamma^{(2)}(t),t)$ are strongly regular. Utilizing the result in Corollary \ref{theo030108}, we obtain
		\begin{equation*}
			\begin{aligned}
				\frac{\partial u(y,v)}{\partial y}\Big\vert_{(y,v)=(\gamma^{(1)}(t),t)}=-\alpha^{(1)}(t),\\
				\frac{\partial u(y,v)}{\partial y}\Big\vert_{(y,v)=(\gamma^{(2)}(t),t)}=-\alpha^{(2)}(t).
			\end{aligned}
		\end{equation*}
		Letting $t\to{s}$, we have $\gamma^{(1)}(t)\to{x}$, $\gamma^{(2)}(t)\to{x}$, and
		\begin{equation*}
			\begin{aligned}
				\frac{\partial u(y,v)}{\partial y}\Big\vert_{(y,v)=(\gamma^{(1)}(t),t)}\to-\alpha^{(1)}(s),\\
				\frac{\partial u(y,v)}{\partial y}\Big\vert_{(y,v)=(\gamma^{(2)}(t),t)}\to-\alpha^{(2)}(s),
			\end{aligned}
		\end{equation*}
		but $\alpha^{(1)}(s)\neq\alpha^{(2)}(s)$ (if they were equal, uniqueness of the solution of the Hamilton equations would force $\gamma^{(1)}=\gamma^{(2)}$). Consequently, $u$ cannot be differentiable at $(x,s)$.
		
		(b) If the minimum of $u(x,s;d,t)$ is uniquely attained but condition (b) of Definition \ref{def030101} fails, then $u(x,s)$ is differentiable at $(x,s)$, yet $\frac{\partial^2u(y,v)}{\partial{y}^2}\to\infty$ as $(y,v)\to(x,s)$, where $(y,v)$ is strongly regular.
		
		Indeed, assume $(x,s)$ satisfies the hypotheses of part (b). By Lemma \ref{theo030106}(c), this situation occurs only when $d_{x,s}=d_1$ in Cases $\text{B}_{I}$, $\text{B}_{II}$, and $\text{B}_{III}$. So, in all these cases, $d_{x,s}+a_0/a_1\neq{0}$. Lemma \ref{theo030106}(b) implies that for every $v\in(s,s+\delta)$ (with some $\delta>0$) the point $(\gamma_{x,s}(v),v)$ is strongly regular. Formula (\ref{eq030108}) then gives
		\begin{equation*}
			\frac{\partial^2 u(y,v)}{\partial y^2}\Big\vert_{(y,v)=(\gamma_{x,s}(v),v)}\neq{0}, \quad\forall\;v\in(s,s+\delta),
		\end{equation*}
		but
		\begin{equation*}
			\frac{\partial^2 u(y,v)}{\partial y^2}\Big\vert_{(y,v)=(\gamma_{x,s}(v),v)}\to\infty,\quad \text{as}\;v\to{s},
		\end{equation*}
		because
		\begin{equation*}
			\frac{\partial g(\gamma_{x,s}(v),v;d_{x,s},t)}{\partial t}\Big\vert_{t=\nu_{x,s}}\to{0},\quad \text{as}\;v\to{s}.
		\end{equation*}
		
		This complements Proposition \ref{theo030107} by illustrating how $u$ loses smoothness at points that are not strongly regular.
	\end{remark}
	\begin{corollary}\label{theo030109}
		(a) Topologically, the set of strongly regular points is open and dense in ${D}\times[0,T)$ for Cases $\text{B}_{I,I}$, $\text{B}_{II}$, $\text{B}_{III}$, and $\text{B}_{IV}$ (and in $\Lambda_3$ for Case $\text{B}_{I,II}$).
		
		(b) In the measure-theoretic sense, almost every point in ${D}\times[0,T)$ for Cases $\text{B}_{I,I}$, $\text{B}_{II}$, $\text{B}_{III}$, and $\text{B}_{IV}$ (and in $\Lambda_3$ for Case $\text{B}_{I,II}$) is strongly regular.
	\end{corollary}
	\begin{proof}
		(a) By Proposition \ref{theo030107}, the set of strongly regular points is open. Let $(x,s)$ be a point in ${D}\times[0,T)$ for Cases $\text{B}_{I,I}$, $\text{B}_{II}$, $\text{B}_{III}$, and $\text{B}_{IV}$ (or in $\Lambda_3$ for Case $\text{B}_{I,II}$). Lemma \ref{theo030106}(b) implies that every open neighborhood of $(x,s)$ contains strongly regular points. Hence, the set is dense. This completes the proof of (a).
		
		(b) Combining Remark \ref{theo030105}(a), the Lipschitz continuity of $u$ (Lemma \ref{theo030105}(h)), and Lemma \ref{theo030106}(c), we conclude that the set of points which are not strongly regular has Lebesgue measure zero. Therefore, almost every point $(x,s)$ with $u(x,s)>0$ is strongly regular, which establishes (b).
	\end{proof}
	\begin{remark}\label{rem030106}Define
		\begin{equation*}
			\begin{aligned}
				u_1(x,s)\triangleq u(x,s;d_1)\triangleq\min_{s<t<T}u(x,s;d_1,t),\\
				u_2(x,s)\triangleq u(x,s;d_2)\triangleq\min_{s<t<T}u(x,s;d_2,t).
			\end{aligned}
		\end{equation*}
		Then
		\begin{equation*}
			u(x,s)=u_1(x,s)\wedge u_2(x,s).
		\end{equation*}
		The following properties can be verified directly.
		
		\textbf{Properties of $u_1$:}
		\begin{itemize}
			\item For every $(x,s)\in{D}\times[0,T)$ in Cases $\text{B}_{I,I}$, $\text{B}_{II}$, $\text{B}_{III}$ (and for every $(x,s)\in\Lambda_3$ in Case $\text{B}_{I,II}$), the equation $g(x,s;d_1,t)=0$ has at least one, and at most three, solutions in $(s,T)$. For every $(x,s)\in{D}\times[0,T)$ in Case $\text{B}_{IV}$, the equation has a unique solution in $(s,T)$. Irrespective of uniqueness, we denote the smallest and largest solutions by $\nu^{(1m)}_{x,s}$ and $\nu^{(1M)}_{x,s}$, respectively. (When the solution is unique, we simply write $\nu^{(1)}_{x,s}$, and it follows that $\nu^{(1)}_{x,s}=\nu^{(1m)}_{x,s}=\nu^{(1M)}_{x,s}$.)
			\item In Cases $\text{B}_{I,I}$, $\text{B}_{I,II}$, $\text{B}_{II}$, $\text{B}_{III}$, when $x\neq x_T$, the minimum of $u(x,s;d_1,t)$ is attained at a unique point in $(s,T)$. Specifically,
			\begin{equation*}
				\qquad\arg\min_{s<t<T} u(x,s;d_1,t)=\begin{cases}
					\nu^{(1m)}_{x,s}<(T+s)/2, &\text{if}\; x<x_T,\\
					\nu^{(1M)}_{x,s}>(T+s)/2, &\text{if}\; x>x_T,\\
				\end{cases}
			\end{equation*}
			and
			\begin{equation*}
				\begin{cases}
					\frac{\partial g(x,s;d_1,t)}{\partial t}\Big\vert_{t=\nu^{(1m)}_{x,s}}<0, &\text{if}\; x<x_T,\\
					\frac{\partial g(x,s;d_1,t)}{\partial t}\Big\vert_{t=\nu^{(1M)}_{x,s}}<0, &\text{if}\; x>x_T.\\
				\end{cases}
			\end{equation*}
			So,
			\begin{equation*}
				u_1(x,s)=\begin{cases}
					u(x,s;d_1,\nu^{(1m)}_{x,s}), &\text{if}\; x<x_T,\\
					u(x,s;d_1,\nu^{(1M)}_{x,s}), &\text{if}\; x>x_T.\\
				\end{cases}
			\end{equation*}
			
			In contrast, the situation when $x=x_T$ exhibits more complex behavior. Let $\tau^{*}$ be the unique time in $(-\infty,T)$ satisfying
			\begin{equation*}
				d_1+\frac{a_0}{a_1}=\left(x_T+\frac{a_0}{a_1}\right)\cosh\left(a_1\frac{T-\tau^{*}}{2}\right).
			\end{equation*}
		    Then
		    \begin{itemize}
		    	\item [$\diamond$] If $s\in(\tau^{*},T)$, the equation $g(x_T,s;d_1,t)=0$ has the unique solution $\nu^{(1)}_{x,s}=(T+s)/2$, and
		    	\begin{equation*}
		    		\frac{\partial g(x_T,s;d_1,t)}{\partial t}\Big\vert_{t=(T+s)/2}<0.
		    	\end{equation*}
		    	\item [$\diamond$] If $s=\tau^{*}$, the equation $g(x_T,s;d_1,t)=0$ has the unique solution $\nu^{(1)}_{x,s}=(T+s)/2$, but
		    	\begin{equation*}
		    		\frac{\partial g(x_T,s;d_1,t)}{\partial t}\Big\vert_{t=(T+s)/2}=0.
		    	\end{equation*}
		    	\item [$\diamond$] If $s\in(-\infty,\tau^{*})$, the equation $g(x_T,s;d_1,t)=0$ has three solutions $\nu^{(1m)}_{x,s}$, $(T+s)/2$, and $\nu^{(1M)}_{x,s}$, where $\nu^{(1m)}_{x,s}$ and $\nu^{(1M)}_{x,s}$ are the two solutions of
		    	\begin{equation*}
		    		\qquad\cosh\left(a_1\left(t-\frac{T+s}{2}\right)\right)=\frac{\cosh\left(a_1\frac{T-s}{2}\right)}{\cosh\left(a_1\frac{T-\tau^{*}}{2}\right)}.
		    	\end{equation*}
		    	The minimum of $u(x_T,s;d_1,t)$ is attained at both $\nu^{(1m)}_{x,s}$ and $\nu^{(1M)}_{x,s}$.
		    \end{itemize}
		    In other words, as $s$ passes through the critical value $\tau^{*}$, the solution set of $g(x_T,s;d_1,t)=0$ undergoes a pitchfork bifurcation. Moreover,
		    \begin{equation*}
		    	\quad u_1(x_T,s)=\begin{cases}
		    		\frac{a_1\sinh\left(a_1\frac{T-s}{2}\right)\left[\left(d_1+\frac{a_0}{a_1}\right)^2-\left(x_T+\frac{a_0}{a_1}\right)^2\right]}{\cosh\left(a_1\frac{T-s}{2}\right)}, &s<\tau^{*},\\
		    		\frac{2a_1\left[\left(d_1+\frac{a_0}{a_1}\right)\cosh\left(a_1\frac{T-s}{2}\right)-\left(x_T+\frac{a_0}{a_1}\right)\right]^2}{\sinh(a_1(T-s))}, &s\geq\tau^{*}.
		    	\end{cases}
		    \end{equation*}
		    This allows us to partition the domain into three distinct subsets:
		    \begin{equation*}
		    	\quad\Theta_1\triangleq\begin{cases}
		    		\{x_T\}\times\left((-\infty,\tau^{*})\cap[0,T)\right),&\text{in Cases}\;\text{B}_{I,I},\text{B}_{II},\text{B}_{III},\\
		    		\left(\{x_T\}\times(-\infty,\tau^{*})\right)\cap\Lambda_3,&\text{in Case}\;\text{B}_{I,II}
		    	\end{cases}
		    \end{equation*}
		    \begin{equation*}
		    	\Theta_2\triangleq\begin{cases}
		    		\{x_T\}\times\left(\{\tau^{*}\}\cap[0,T)\right),&\text{in Cases}\;\text{B}_{I,I},\text{B}_{II},\text{B}_{III},\\
		    		\{(x_T,\tau^{*})\}\cap\Lambda_3,&\text{in Case}\;\text{B}_{I,II}
		    	\end{cases}
		    \end{equation*}
		    \begin{equation*}
		    	\Theta_3\triangleq\begin{cases}
		    		{D}\times[0,T)\setminus\left(\Theta_1\cup\Theta_2\right),&\text{in Cases}\;\text{B}_{I,I},\text{B}_{II},\text{B}_{III},\\
		    		\Lambda_3\setminus\left(\Theta_1\cup\Theta_2\right),&\text{in Case}\;\text{B}_{I,II}
		    	\end{cases}
		    \end{equation*}
		    (Note that $\Theta_1$ and $\Theta_2$ are non-empty only if $x_T$ is sufficiently close to $d_1$ or $T$ is sufficiently large, while $\Theta_3$ is always non-empty.) Consequently, we obtain
		    \begin{itemize}
		    	\item [$\diamond$] $u_1$ is of class $\mathcal{C}^{\infty}$ on $\Theta_3$.
		    	\item [$\diamond$] $u_1$ is non-differentiable at every point of $\Theta_1$.
		    	\item [$\diamond$] $u_1$ is differentiable at points of $\Theta_2$, but $\frac{\partial^2 u_1(y,v)}{\partial y^2}\to\infty$ as $(y,v)\to(x,s)$ for $(x,s)\in\Theta_2$ and $(y,v)\in\Theta_3$.		    	
		    \end{itemize}
		    \item In Case $\text{B}_{IV}$, the minimum of $u(x,s;d_1,t)$ is attained at the unique solution $\nu^{(1)}_{x,s}$, which satisfies
		    \begin{equation*}
		    	\frac{\partial g(x,s;d_1,t)}{\partial t}\Big\vert_{t=\nu^{(1)}_{x,s}}<0,
		    \end{equation*}
		    and
		    \begin{equation*}
		    	\nu^{(1)}_{x,s}\begin{cases}
		    		<(T+s)/2, &\text{if}\; x<x_T,\\
		    		=(T+s)/2, &\text{if}\; x=x_T,\\
		    		>(T+s)/2, &\text{if}\; x>x_T.\\
		    	\end{cases}
		    \end{equation*}
		    Hence, $u_1(x,s)=u(x,s;d_1,\nu^{(1)}_{x,s})$ is of class $\mathcal{C}^{\infty}$ on its domain. In particular, when $x=x_T$,
		    \begin{equation*}
		    		u_1(x_T,s)=\frac{2a_1\left[\left(d_1+\frac{a_0}{a_1}\right)\cosh\left(a_1\frac{T-s}{2}\right)\right]^2}{\sinh(a_1(T-s))}.
		    \end{equation*}
		\end{itemize}
		
		\textbf{Properties of $u_2$:}
		\begin{itemize}
			\item For every $(x,s)\in{D}\times[0,T)$ in Cases $\text{B}_{I,I}$, $\text{B}_{II}$, $\text{B}_{III}$, $\text{B}_{IV}$ (and for every $(x,s)\in\Lambda_3$ in Case $\text{B}_{I,II}$), the equation $g(x,s;d_2,t)=0$ has a unique solution $\nu^{(2)}_{x,s}$ in $(s,T)$.
			\item This solution satisfies 
			\begin{equation*}
				\frac{\partial g(x,s;d_2,t)}{\partial t}\Big\vert_{t=\nu^{(2)}_{x,s}}>0,
			\end{equation*}
			and
			\begin{equation*}
					\nu^{(2)}_{x,s}\begin{cases}
							>(T+s)/2, &\text{if}\; x<x_T,\\
							=(T+s)/2, &\text{if}\; x=x_T,\\
							<(T+s)/2, &\text{if}\; x>x_T.\\
						\end{cases}
			\end{equation*}
			Consequently, $u_2(x,s)=u(x,s;d_2,\nu^{(2)}_{x,s})$ is of class $\mathcal{C}^{\infty}$ on its domain in all cases. In particular, when $x=x_T$,
			\begin{equation*}
					u_2(x_T,s)=\frac{2a_1\left[\left(d_2+\frac{a_0}{a_1}\right)\cosh\left(a_1\frac{T-s}{2}\right)-\left(x_T+\frac{a_0}{a_1}\right)\right]^2}{\sinh(a_1(T-s))}.
			\end{equation*}
		\end{itemize}
		
		\textbf{Properties of $u$:}
		\begin{itemize}
			\item $u$ is of class $\mathcal{C}^{\infty}$ at every strongly regular point, i.e., at every point $(x,s)$ that satisfies
			\begin{equation*}
				u_1(x,s)>u_2(x,s),
			\end{equation*}
			or
			\begin{equation*}
				u_1(x,s)<u_2(x,s)\quad\text{and}\quad (x,s)\in\Theta_3.
			\end{equation*}
			\item $u$ is non-differentiable at every point $(x,s)$ satisfying
			\begin{equation*}
				u_1(x,s)=u_2(x,s),
			\end{equation*}
			or
			\begin{equation*}
				u_1(x,s)<u_2(x,s)\quad\text{and}\quad (x,s)\in\Theta_1.
			\end{equation*}
			(The later situation occurs only when $x_T$ is sufficiently close to $d_1$ or $T$ is sufficiently large.)
			\item $u$ is differentiable at every point $(x,s)$ satisfying
			\begin{equation*}
				u_1(x,s)<u_2(x,s)\quad\text{and}\quad (x,s)\in\Theta_2,
			\end{equation*}
			but $\frac{\partial^2 u(y,v)}{\partial y^2}\to\infty$ as $(y,v)\to(x,s)$, where $(y,v)$ is strongly regular. (This situation occurs only when $x_T$ is sufficiently close to $d_1$ or $T$ is sufficiently large.)
		\end{itemize}
		
		This provides further clarification of the precise locations of the points that are not strongly regular.
	\end{remark}
	
	Applying the Fleming-James theorem, we then obtain the following result.
	\begin{proposition}\label{theo030110}
		Let $S\subset{D}\times[0,T)$ be the set of all strongly regular points. Then for each $m=0,1,2,\cdots$, we have
		\begin{equation}\label{eq030109}
			\begin{aligned}
				q^{\varepsilon}(x,s)=&\exp\left(-\frac{u(x,s)}{\varepsilon}-w(x,s)\right)(1+\varepsilon\psi_1(x,s)\\
				&+\cdots+\varepsilon^{m}\psi_m(x,s)+o(\varepsilon^{m})),
			\end{aligned}
		\end{equation}
		as $\varepsilon\to{0}$, uniformly on compact subsets of every connected component of $S$. Here, the functions $w$, $\psi_m$ are $\mathcal{C}^{\infty}$, satisfy Equations (\ref{eqA10108}) and (\ref{eqA10109}), and can be calculated by the method of characteristics (see Remark A.1).
	\end{proposition}
	\begin{proof}
		For any $T'\in(0,T)$, define
		\begin{equation*}
			S_{T'}\triangleq\{(x,s)\in{S}:\nu_{x,s}\in(0,T')\}.
		\end{equation*}
		By Lemma \ref{theo030106}(a), Proposition \ref{theo030107}, and Corollary \ref{theo030108}, every connected component of $S_{T'}$ satisfies Assumption (A). Moreover, Lemma \ref{theo030105}(e) shows that for every compact connected subset $K$ of $S$, there exists some $T'\in(0,T)$ such that $\nu_{x,s}<T'$ for all $(x,s)\in K$. Hence, $K\subset S_{T'}$, and so $K$ is contained in a connected component of $S_{T'}$. The statement then follows directly from the Fleming-James theorem. This completes the proof. 
	\end{proof}
	\begin{remark}\label{rem030107}
		It is evident that the equation $g(x,s;d,t)=0$ is of quartic type, which makes it difficult to find explicit solutions. Consequently, our ability to provide closed-form expressions is limited, and we must restrict ourselves to qualitative descriptions of the quantities involved. Fortunately, in the special case where $a_0=0$ and $a_1\to{0}$, the problem simplifies considerably and all quantities can be determined explicitly. This situation will be discussed in detail in the next section.
	\end{remark}
	
	\section{In The limiting case: $a_0=0$, $a_1\to{0}$}\label{Sec04}
	When $a_0=0$ and $a_1\to{0}$, the OU bridge reduces to a Brownian bridge, which is governed by the SDE (\ref{eq010101}) with drift term
	\begin{equation*}
		b(x,t)=\frac{x_T-x}{T-t}.
	\end{equation*}
	The asymptotic series of $q^{\varepsilon}$ for this process follows as corollaries of Proposition \ref{theo0205}, \ref{theo030103}, \ref{theo030110}, and is given below.
	\begin{corollary}\label{theo040101}
		(a) For any $t\in[s,T]$, $x^{0}_{x,s}(t)=x\frac{T-t}{T-s}+x_T\frac{t-s}{T-s}$. 
		
		(b) So, $x^{0}_{x,s}$ is monotonic on $[s,T]$ for all $(x,s)\in\mathbb{R}\times[0,T)$. 
		
		(c) The set $\Omega$ is always empty. 
		
		(d)  When $x_T>d_2>d_1$, the trajectory $\{(x^{0}_{x,s}(t),t):t\in[s,T]\}$ intersects $\partial{D}\times[0,T)$ transversely at a unique point $(d_2,\tau^{0}_{x,s})$ with $\tau^{0}_{x,s}=s\frac{x_T-d_2}{x_T-x}+T\frac{d_2-x}{x_T-x}$. If $x_T\in{D}$, the trajectory remains inside ${D}\times[0,T)$ until it reaches $(x_T,T)$.
	\end{corollary}
	\begin{corollary}\label{theo040102}
		(a) If $x_T\notin{D}$, then $q^{\varepsilon}(x,s)\equiv{1}$ for all $(x,s)\in{D}\times[0,T)$.
		
		(b) In particular, when $x_T>d_2>d_1$, for every $(x,s)\in{D}\times[0,T)$ and $\eta>0$,
		\begin{equation*}
			\lim_{\varepsilon\to{0}}\mathbb{P}^{\varepsilon}_{x,s}\left(\left\vert\tau^{\varepsilon}_{x,s}-\tau^{0}_{x,s}\right\vert>\eta\right)={0},
		\end{equation*}
		\begin{equation*}
			\lim_{\varepsilon\to{0}}\mathbb{P}^{\varepsilon}_{x,s}\left({x}^{\varepsilon}_{\tau^{\varepsilon}_{x,s}}=d_2\right)=1.
		\end{equation*}
	\end{corollary}
	\begin{corollary}\label{theo040103}
		Assume $x_T\in{D}$. Then the following hold:
		
		(a) The function $u(x,s;d,t)$ can be expressed as
		\begin{equation*}
			\begin{aligned}
				u(x,s;d,t)=&\frac{(d(T-s)-x(T-t)-x_T(t-s))^2}{2(T-s)(T-t)(t-s)}\\
				=&\frac{1}{2}\left(\frac{(d-x)^2}{t-s}+\frac{(d-x_T)^2}{T-t}-\frac{(x-x_T)^2}{T-s}\right).
			\end{aligned}
		\end{equation*}
		
		(b) The equation $\frac{\partial u(x,s;d,t)}{\partial t}=0$ has a unique solution 
		\begin{equation*}
			\nu^{*}=s+(T-s)\frac{x-d}{x+x_T-2d}.
		\end{equation*}
		
		(c) Consequently,
		\begin{equation*}
			d_{x.s}=\begin{cases}
				d_1, & \text{if}\;x\leq d_1+d_2-x_T,\\
				d_2, & \text{if}\;x\geq d_1+d_2-x_T,
			\end{cases}
		\end{equation*}
		\begin{equation*}
			\nu_{x,s}=\begin{cases}
				s+(T-s)\frac{x-d_1}{x+x_T-2d_1}, & \text{if}\;x\leq d_1+d_2-x_T,\\
				s+(T-s)\frac{x-d_2}{x+x_T-2d_2}, & \text{if}\;x\geq d_1+d_2-x_T,
			\end{cases}
		\end{equation*}
		\begin{equation*}
			u(x,s)=\begin{cases}
				\frac{2(x-d_1)(x_T-d_1)}{T-s}, & \text{if}\;x\leq d_1+d_2-x_T,\\
				\frac{2(d_2-x)(d_2-x_T)}{T-s}, & \text{if}\;x\geq d_1+d_2-x_T.
			\end{cases}
		\end{equation*}
		
		(d) The set of all strongly regular points is
		\begin{equation*}
			S=\left((d_1,d_1+d_2-x_T)\cup(d_1+d_2-x_T,d_2)\right)\times[0,T).
		\end{equation*}
		
		(e) For any $(x,s)\in{D}\times[0,T)\setminus S$, the minimum of $u(x,s;d,t)$ is attained at two distinct points $(d_1,\frac{(x_T-d_1)s+(d_2-x_T)T}{d_2-d_1})$ and $(d_2,\frac{(d_2-x_T)s+(x_T-d_1)T}{d_2-d_1})$, giving two coexisting optimal paths:
		\begin{equation*}
			\begin{aligned}
				\gamma_{x,s}^{(1)}=&\gamma_{x,s;d_1,\frac{(x_T-d_1)s+(d_2-x_T)T}{d_2-d_1}},\\
				\gamma_{x,s}^{(2)}=&\gamma_{x,s;d_2,\frac{(d_2-x_T)s+(x_T-d_1)T}{d_2-d_1}}.\\
			\end{aligned}
		\end{equation*}
		There is no point for which condition (a) of Definition \ref{def030101} holds but condition (b) fails.
		
		(f) Since $\frac{\partial^2 u}{\partial x^2}\equiv 0$ on $S$, we have $w(x,s)\equiv{0}$ and $\psi_m(x,s)\equiv{0}$ for all $(x,s)\in{S}$ and every $m>1$.
		
		(g) For each $m=0,1,2,\cdots$,
		\begin{equation}\label{eq040101}
			q^{\varepsilon}(x,s)=\exp\left(-\frac{u(x,s)}{\varepsilon}\right)(1+o(\varepsilon^{m})),\quad \text{as}\;\varepsilon\to{0},
		\end{equation}
		uniformly on compact subsets of $S$.
	\end{corollary}
	\begin{remark}\label{rem040101}
		Similar to Remark \ref{rem030101}, the infinitesimal term $o(\varepsilon^m)$ in (\ref{eq040101}) cannot be neglected either. Observe that
		\begin{equation*}
			\begin{aligned}
				\mathbb{P}^{\varepsilon}_{x,s}\left(\tau^{\varepsilon}_{x,s}\leq{T}\right)=&\mathbb{P}^{\varepsilon}_{x,s}\left(\tau^{\varepsilon}_{x,s}\leq{T},\,x^{\varepsilon}_{\tau^{\varepsilon}_{x,s}}=d_1\right)\\
				&+\mathbb{P}^{\varepsilon}_{x,s}\left(\tau^{\varepsilon}_{x,s}\leq{T},\,x^{\varepsilon}_{\tau^{\varepsilon}_{x,s}}=d_2\right).
			\end{aligned}
		\end{equation*}
		Large deviation theory (cf. [\onlinecite[Chapter 3, Theorem 3.4]{Freidlin_2012}]) states that the asymptotics of the two terms on the right are
		\begin{equation*}
			\begin{aligned}
				\mathbb{P}^{\varepsilon}_{x,s}\left(\tau^{\varepsilon}_{x,s}\leq{T},\,x^{\varepsilon}_{\tau^{\varepsilon}_{x,s}}=d_1\right)\simeq&\exp\left(-\frac{u_1(x,s)}{\varepsilon}\right),\\
				\mathbb{P}^{\varepsilon}_{x,s}\left(\tau^{\varepsilon}_{x,s}\leq{T},\,x^{\varepsilon}_{\tau^{\varepsilon}_{x,s}}=d_2\right)\simeq&\exp\left(-\frac{u_2(x,s)}{\varepsilon}\right),
			\end{aligned}
		\end{equation*}
		where $u_1$ and $u_2$ (defined in Remark \ref{rem030106}) are given by
		\begin{equation*}
			\begin{aligned}
				u_1(x,s)=&\frac{2(x-d_1)(x_T-d_1)}{T-s},\\
				u_2(x,s)=&\frac{2(d_2-x)(d_2-x_T)}{T-s}.
			\end{aligned}
		\end{equation*}
		
		If $u_1(x,s)\neq u_2(x,s)$ (i.e., $x\neq d_1+d_2-x_T$), then
		\begin{equation*}
			\mathbb{P}^{\varepsilon}_{x,s}\left(\tau^{\varepsilon}_{x,s}\leq{T}\right)\simeq\exp\left(-\frac{u(x,s)}{\varepsilon}\right)\left(1+\exp\left(-\frac{\Delta u}{\varepsilon}\right)\right),
		\end{equation*}
		with
		\begin{equation*}
			\Delta u\triangleq u_1\vee u_2-u_1\wedge u_2.
		\end{equation*}
		Since $\exp\left(-{\Delta u}/{\varepsilon}\right)$ is smooth but non-analytic in $\varepsilon$ at $\varepsilon=0$, it cannot be expanded in a convergent Taylor series. Consequently, although Equation (\ref{eq040101}) holds for every $m\geq{0}$, the term $o(\varepsilon^{m})$ cannot be neglected, as it represents an exponentially small contribution. 
		
		In contrast, when $x=d_1+d_2-x_T$, we have $u_1(x,s)=u_2(x,s)$. Therefore,
		\begin{equation*}
			\mathbb{P}^{\varepsilon}_{x,s}\left(\tau^{\varepsilon}_{x,s}\leq{T}\right)\simeq2\exp\left(-\frac{u(x,s)}{\varepsilon}\right).
		\end{equation*}
		This explains why the asymptotic series (\ref{eq040101}) fails at points that are not strongly regular.
	\end{remark}
	
	\section{Conclusions}\label{Sec05}
	 In this paper, we have derived an accurate asymptotic expansion for the exit time probabilities of scalar Ornstein-Uhlenbeck bridges. By analyzing the solution of an associated constrained optimization problem, we have shown that the asymptotic series remains valid for generic positions, both in the topological and in the measure‑theoretic sense. The results presented here extend earlier findings on diffusion bridges and provide a precise estimate of exit time probabilities for the corresponding Ornstein-Uhlenbeck process conditioned on prescribed initial and terminal states.

	\begin{acknowledgments}
		The authors acknowledge support from the National Natural Science Foundation of China (Nos. 12172167, 12202318, 12302035, 12572038) and the Jiangsu Funding Program for Excellent Postdoctoral Talent (No. 2023ZB591).
	\end{acknowledgments}
	
	\section*{Data Availability Statement}
	The data that support the findings of this study are available from the corresponding author upon reasonable request. 
		
	\appendix
	\section{Fleming-James Theorem for Diffusion Processes}\label{SecA1}
	Let $D\subset\mathbb{R}$ be a non-empty open set. Fix $T>0$ and consider a diffusion process $(x^{\varepsilon},\mathbb{P}^{\varepsilon}_{x,s})$ governed by the following SDE:
	\begin{equation}\label{eqA10101}
		\begin{aligned}
			&\mathrm{d}x^{\varepsilon}_{t}=b(x^{\varepsilon}_{t},t)\mathrm{d}{t}+\sqrt{\varepsilon}\mathrm{d}w_{t},\quad t\in(s,T],\\ &x^{\varepsilon}_{s}=x\in{D}.
		\end{aligned}
	\end{equation}
	The drift coefficient $b$ is assumed to be of class $\mathcal{C}^{\infty}(\mathbb{R}\times[0,T],\mathbb{R})$ and Lipschitz continuous in $x$, uniformly with respect to $t$. Let $\tau^{\varepsilon}_{x,s}\triangleq\inf\{t>s:x^{\varepsilon}_{t}\notin{D}\}$ be the first exit time of $x^{\varepsilon}$ from $D$. Define the exit time probability function as
	\begin{equation*}
		q^{\varepsilon}(x,s)\triangleq \mathbb{P}^{\varepsilon}_{x,s}\left(\tau^{\varepsilon}_{x,s}\leq{T}\right).
	\end{equation*}
	Then $q^{\varepsilon}(x,s)$ belongs to $\mathcal{C}^{2,1}(K)$ for every compact subset $K\subset\bar{D}\times[0,T)$, and solves the boundary value problem \cite{Fleming_1992}:
	\begin{equation}\label{eqA10102}
		\begin{aligned}
			&\frac{\partial{q}^{\varepsilon}}{\partial{s}}+b(x,s)\frac{\partial{q}^{\varepsilon}}{\partial{x}}+\frac{\varepsilon}{2}\frac{\partial^2{q}^{\varepsilon}}{\partial{x}^2}=0\quad\text{in}\;D\times(0,T),\\
			&{q}^{\varepsilon}(x,s)=1\quad\text{on}\; \partial{D}\times[0,T),\\
			&{q}^{\varepsilon}(x,T)=0\quad \text{if}\;x\in{D}.
		\end{aligned}
	\end{equation}
	
	Define the set of admissible paths by
	\begin{equation*}
		\Gamma_{x,s}\triangleq\{\gamma\in\mathcal{C}([s,T],\mathbb{R}):\gamma(s)=x, \exists\,t\in(s,T]\, \text{s.t.}\, \gamma(t)\in\partial{D}\},
	\end{equation*}
	and let
	\begin{equation}\label{eqA10103}
		u(x,s)\triangleq\inf_{\gamma\in{\Gamma}_{x,s}}I(\gamma)\triangleq\inf_{\gamma\in{\Gamma}_{x,s}}\frac{1}{2}\int_{s}^{T}[\dot{\gamma}(t)-b(\gamma(t),t)]^2\mathrm{d}t.
	\end{equation}
	It can be shown that $u\in\mathcal{C}(\bar{D}\times[0,T))$ is the unique viscosity solution of the Hamilton-Jacobi equation \cite{Fleming1986b}:
	\begin{equation}\label{eqA10104}
		\begin{aligned}
			&\frac{\partial{u}}{\partial{s}}+b(x,s)\frac{\partial{u}}{\partial{x}}-\frac{1}{2}\left(\frac{\partial{u}}{\partial{x}}\right)^2=0\quad\text{in}\;D\times(0,T),\\
			&u(x,s)=0\quad \text{on}\;\partial{D}\times[0,T],\\
			&u(x,s)\to\infty\quad\text{as}\;s\to{T}\;\text{if}\;x\in{D}.
		\end{aligned}
	\end{equation}
	The leading-order asymptotics are given by the following well-known result \cite{Wentzell_1970,Freidlin_2012}.
	\begin{theorem}\label{theoA101}
		We have
		\begin{equation*}
			\lim_{\varepsilon\to{0}}\varepsilon\ln{q}^{\varepsilon}(x,s)=-u(x,s),
		\end{equation*}
		uniformly on compact subsets of $\bar{D}\times[0,T)$.
	\end{theorem}
	The next result, due to Fleming and James (cf. [\onlinecite[Theorem 4.2]{Fleming_1992}]), provides an asymptotic expansion for ${q}^{\varepsilon}(x,s)$, at least for points $(x,s)$ in a set $N\subset{D}\times[0,T)$ that satisfies the following assumptions.
	
	\textbf{Assumption (A):}
	\begin{itemize}
		\item There exists $T'<T$ such that $N\subset{D}\times[0,T']$ and $N$ is open;
		\item $u\in\mathcal{C}^{\infty}(\bar{N})$;
		\item Define
		\begin{equation}\label{eqA10105}
			\beta(x,s)\triangleq b(x,s)-\frac{\partial{u}(x,s)}{\partial{x}},\quad(x,s)\in\bar{N}.
		\end{equation}
		For each $(x,s)\in N$, let $\gamma_{x,s}$ be the solution of the ODE:
		\begin{equation}\label{eqA10106}
			\begin{aligned}
				&\dot{\gamma}_{x,s}(t)=\beta(\gamma_{x,s}(t),t),\\
				&\gamma_{x,s}(s)=x.
			\end{aligned}
		\end{equation}
		Denote by $\nu_{x,s}$ the first time at which the trajectory $(\gamma_{x,s}(t),t)$ reaches the boundary $\partial{N}$. Set
		\begin{equation*}
			\Gamma_1\triangleq\{z_{x,s}=\left(\gamma_{x,s}(\nu_{x,s}),\nu_{x,s}\right):(x,s)\in{N}\}\subset\partial{N}.
		\end{equation*}
		The following conditions are imposed:
		\begin{itemize}
			\item [$\diamond$] $\Gamma_1\subset\partial{D}\times(0,T')$;
			\item [$\diamond$] $\Gamma_1$ is a $\mathcal{C}^{\infty}$-manifold and is relatively open in $\partial{N}$;
			\item [$\diamond$] $(\gamma_{x,s}(t),t)$ crosses $\partial{N}$ non-tangentially.
		\end{itemize}
	\end{itemize}
	\begin{theorem}\label{theoA102}
		Suppose that $N\subset{D}\times[0,T)$ satisfies Assumption (A). Then for every $m=0,1,2,\cdots$,
		\begin{equation}\label{eqA10107}
			\begin{aligned}
				q^{\varepsilon}(x,s)=&\exp\left(-\frac{u(x,s)}{\varepsilon}-w(x,s)\right)(1+\varepsilon\psi_1(x,s)\\
				&+\cdots+\varepsilon^{m}\psi_m(x,s)+o(\varepsilon^{m})),
			\end{aligned}
		\end{equation}
		as $\varepsilon\to{0}$, uniformly on compact subsets of $N$. The functions $w,\psi_m\in\mathcal{C}^{\infty}(N)$ (with $\psi_0\equiv{1}$) satisfy
		\begin{equation}\label{eqA10108}
			\begin{aligned}
				&\frac{\partial{w}}{\partial{s}}+\left(b-\frac{\partial{u}}{\partial{x}}\right)\frac{\partial{w}}{\partial{x}}=-\frac{1}{2}\frac{\partial^2{u}}{\partial{x}^2}\quad\text{in}\;N,\\
				&w(x,s)=0\quad\text{on}\;\partial{D}\times[0,T)\cap\bar{N},
			\end{aligned}
		\end{equation}
		and, for $m\geq 1$,
		\begin{equation}\label{eqA10109}
			\begin{aligned}
				\frac{\partial{\psi_m}}{\partial{s}}&+\left(b-\frac{\partial{u}}{\partial{x}}\right)\frac{\partial{\psi_m}}{\partial{x}}=-\Bigg[\frac{1}{2}\Bigg(\bigg(\frac{\partial{w}}{\partial{x}}\bigg)^2-\frac{\partial^2{w}}{\partial{x}^2}\Bigg)\\
				&\times\psi_{m-1}-\frac{\partial{w}}{\partial{x}}\frac{\partial{\psi_{m-1}}}{\partial{x}}+\frac{1}{2}\frac{\partial^2{\psi_{m-1}}}{\partial{x}^2}\Bigg]\quad\text{in}\;N,\\
				\psi_m(x&,s)=0\quad\text{on}\;\partial{D}\times[0,T)\cap\bar{N}.
			\end{aligned}
		\end{equation}
	\end{theorem}
	\begin{remark}\label{remA101}
			(a) Essentially, Assumption (A) provides a strong regularity condition that ensures the validity of the asymptotic expansion for $q^{\varepsilon}$. If a region $N$ satisfies this assumption, then the following properties hold:
			\begin{itemize}
				\item For every $(x,s)\in{N}$, the infimum in (\ref{eqA10103}) is attained by a unique path $\gamma\in{\Gamma_{x,s}}$.
				\item Classical calculus of variations shows that this minimizing path coincides with the characteristic curve $\gamma_{x,s}$ defined by (\ref{eqA10106}), i.e., $\gamma(t)=\gamma_{x,s}(t)$ for all $t\in[s,\nu_{x,s}]$. Moreover, the trajectory $(\gamma_{x,s}(t),t)$ first exits $N$ at time $\nu_{x,s}<T'$, and its exit point $(\gamma_{x,s}(\nu_{x,s}),\nu_{x,s})$ belongs to $\Gamma_1\subset\partial{D}\times(0,T')$. Consequently,
				\begin{equation*}
					u(x,s)=\frac{1}{2}\int_{s}^{\nu_{x,s}}[\dot{\gamma}_{x,s}(t)-b(\gamma_{x,s}(t),t)]^2\mathrm{d}t.
				\end{equation*}
				\item Equation (\ref{eqA10106}) is the characteristic differential equation for the PDEs (\ref{eqA10108}) and (\ref{eqA10109}). Applying the method of characteristics yields explicit integral representations:
				\begin{equation*}
					w(x,s)=\frac{1}{2}\int_{s}^{\nu_{x,s}}\frac{\partial^2{u}}{\partial{x}^2}(\gamma_{x,s}(t),t)\mathrm{d}t,
				\end{equation*}
				and for $m\geq 1$,
				\begin{equation*}
					\begin{aligned}
						\psi_m(x,s)=&\int_{s}^{\nu_{x,s}}\Bigg[\frac{1}{2}\Bigg(\bigg(\frac{\partial{w}}{\partial{x}}\bigg)^2-\frac{\partial^2{w}}{\partial{x}^2}\Bigg)\psi_{m-1}-\\
						&\frac{\partial{w}}{\partial{x}}\frac{\partial{\psi_{m-1}}}{\partial{x}}+\frac{1}{2}\frac{\partial^2{\psi_{m-1}}}{\partial{x}^2}\Bigg](\gamma_{x,s}(t),t)\mathrm{d}t.
					\end{aligned}
				\end{equation*}
			\end{itemize}
			 
			(b) In general, the asymptotic expansion in Theorem \ref{theoA102} may fail in regions where Assumption (A) is not satisfied. For a discussion of this phenomenon, see Remark 2.3 of [\onlinecite{Baldi_1995}].
			
			(c) Extensions of Theorem \ref{theoA102} to situations where the drift $b$ depends on $\varepsilon$ or the diffusion coefficient is non-constant can be found in [\onlinecite{Fleming_1992}].
	\end{remark}
	
	\section{Relationship with the Conditional Exit Time Probability of the Corresponding Ornstein-Uhlenbeck Process}\label{SecA2}
	Let $\{\tilde{x}^{\varepsilon}_t\}_{t\geq{0}}$ be a one-dimensional Ornstein-Uhlenbeck process governed by the stochastic differential equation (SDE):
	\begin{equation}
		\begin{aligned}
			&\mathrm{d}\tilde{x}^{\varepsilon}_t=\left(a_0+a_1\tilde{x}^{\varepsilon}_t\right)\mathrm{d}t+\sqrt{\varepsilon}\mathrm{d}w_t,\quad t>0,\\
			&\tilde{x}^{\varepsilon}_0=x_0.
		\end{aligned}
	\end{equation}
	Clearly, the quantity $-a_0/a_1$ represents the fixed point of the vector field $\tilde{b}(x)=a_0+a_1x$. It is well known that $\{\tilde{x}^{\varepsilon}_t\}_{t\geq{0}}$ is a Gauss-Markov process with mean $\mathbb{E}^{\varepsilon}(\tilde{x}^{\varepsilon}_t)=\left(x_0+\frac{a_0}{a_1}\right)e^{a_1t}-\frac{a_0}{a_1}$ and covariance $\mathbb{C}\text{ov}\left(\tilde{x}^{\varepsilon}_s,\tilde{x}^{\varepsilon}_t\right)=\frac{\varepsilon{e}^{a_1t}\sinh(a_1s)}{a_1}$ for $0\leq{s}\leq{t}$.
	
	An Ornstein-Uhlenbeck bridge $x^{\varepsilon}_t$ with prescribed terminal condition $x^{\varepsilon}_T=x_T$ (associated with the OU process $\tilde{x}^{\varepsilon}_t$) is a stochastic process defined by
	\begin{equation*}
		\text{Law}(x^{\varepsilon},\mathbb{P}^{\varepsilon})=\text{Law}(\tilde{x}^{\varepsilon},\tilde{\mathbb{P}}^{\varepsilon}_{x_0,0;x_T,T}),
	\end{equation*}
	where
	\begin{equation*}
		\tilde{\mathbb{P}}^{\varepsilon}_{x_0,0;x_T,T}\triangleq\mathbb{P}^{\varepsilon}_{x_0,0}\left(\cdot\vert\tilde{x}^{\varepsilon}_T=x_T\right)=\mathbb{P}^{\varepsilon}\left(\cdot\vert\tilde{x}^{\varepsilon}_0=x_0,\tilde{x}^{\varepsilon}_T=x_T\right).
	\end{equation*}
	As is well known (see, e.g., [\onlinecite{Barczy_2013,Azze_2024,Azze_2025}]), the process $x^{\varepsilon}_t$ coincides with the unique strong solution of the SDE (\ref{eq010101}) endowed with the drift term (\ref{eq010102}) and the initial condition $x^{\varepsilon}_0=x_0$.
	
	Let $\tilde{\tau}^{\varepsilon}_{x_0,0}\triangleq\inf\{t>0:\tilde{x}^{\varepsilon}_t\notin{D}\}$ denote the first exit time of $\tilde{x}^{\varepsilon}$ from $D$, and define the conditional exit time probability by
	\begin{equation*}
		\tilde{q}^{\varepsilon}(x_0,0)\triangleq\mathbb{P}^{\varepsilon}_{x_0,0}\left(\tilde{\tau}^{\varepsilon}_{x_0,0}\leq{T}\vert\tilde{x}^{\varepsilon}_T=x_T\right).
	\end{equation*}
	Then, in view of the relationship between $\tilde{x}^{\varepsilon}$ and $x^{\varepsilon}$, we obtain the identity
	\begin{equation*}
		\tilde{q}^{\varepsilon}(x_0,0)=q^{\varepsilon}(x_0,0),\quad\forall\;x_0\in{D}.
	\end{equation*}
	Therefore, the accurate asymptotic expansion of $\tilde{q}^{\varepsilon}$ in powers of $\varepsilon$ follows directly from Proposition \ref{theo0205}, \ref{theo030103}, and \ref{theo030110}.

	\nocite{*}
	\bibliography{Rerferences}
		
	\end{document}